\documentclass[smallextended,envcountsect]{svjour3}
\smartqed
\usepackage[margin=25mm]{geometry}
\usepackage{enumerate}
\usepackage{amsmath,amssymb,amsfonts}
\usepackage{mathtools}
\usepackage{graphicx}
\usepackage{color}
\usepackage{subfig}
\usepackage{url}
\usepackage{hyperref}
\usepackage{booktabs}
\usepackage{multirow}

\newcommand{\st}{{\rm s.~t.}}

\newcommand{\vt}[1]{\boldsymbol{#1}}
\newcommand{\bb}[1]{\mathbb{#1}}

\makeatletter
\@ifundefined{spnewtheorem}%
  {\newtheorem{assumption}[theorem]{Assumption}}%
  {\spnewtheorem{assumption}[theorem]{Assumption}{\bfseries}{\itshape}}
\makeatother

\journalname{}

\begin{document}

\title{A Unifying View of Distributionally Robust Optimization for Probability Smoothing via Maximum Entropy}

\titlerunning{DRO Smoothing via Maximum Entropy}        

\author{Kota~Kurihara \and Yuya~Yamakawa \and Yoichi~Izunaga}

\institute{Kota~Kurihara \at
              Kyushu University \\
              Fukuoka, Japan\\
              kurihara.kota.881@s.kyushu-u.ac.jp
           \and
              Yuya Yamakawa \at
              Tokyo Metropolitan University \\
              Tokyo, Japan\\
              yuya@tmu.ac.jp
            \and
              Yoichi Izunaga \at
              Kyushu University \\
              Fukuoka, Japan\\
              izunaga@econ.kyushu-u.ac.jp
}

\date{Received: date / Accepted: date}

\maketitle

\begin{abstract}
The probability smoothing problem seeks estimators that assign positive probability to all categories, including those unobserved in a finite sample.
In our previous work, we established that the estimator derived from distributionally robust optimization (DRO) with a $q$-norm ambiguity set ($q$-DRO) satisfies the axioms of Positivity, Symmetry, and Order Preservation, while relaxing the rigid Ratio Preservation axiom characterizing Laplace smoothing.
In this paper, we observe that Symmetry itself imposes a strong constraint: it forces all unobserved categories to receive identical probability regardless of their relation to observed categories.
Motivated by this observation, we study DRO with an optimal transport ambiguity set (OT-DRO), which incorporates inter-category dissimilarities through a cost matrix and thereby enables structure-respecting smoothing.
We derive a tractable dual reformulation of OT-DRO as an empirical cross-entropy minimization with a Lipschitz-type penalty on the log-probability vector, and we identify an explicit threshold on the robustness radius at which the optimal estimator reduces to the uniform distribution, paralleling the threshold established for $q$-DRO.
We then show, through Sion's minimax theorem, that DRO for probability smoothing over any compact convex ambiguity set that contains a strictly positive distribution is equivalent to entropy maximization over that set.
This unifying perspective explains the parallel threshold structure in $q$-DRO and OT-DRO.
\end{abstract}

\section{Introduction}
\label{sec:intro}
The zero-frequency problem arises whenever probabilities must be estimated from a finite sample in which some categories are unobserved.
A classical remedy is Laplace smoothing, which assigns a uniform pseudocount to every category.
Sakai~\cite{S_MSS2025} recently characterized Laplace smoothing as the unique estimator satisfying four axioms: Positivity, Symmetry, Order Preservation, and Ratio Preservation. While this characterization provides a principled foundation, the Ratio Preservation axiom imposes a strong linear structure that can be inappropriate for complex applications.

A flexible alternative is distributionally robust optimization (DRO), in which the estimator is chosen to minimize the worst-case loss over an ambiguity set of distributions centered at the empirical distribution.
The duality theory and tractable reformulations of DRO have been developed extensively~\cite{BN_RO2009,KSW_DRO2025,MK_MM2018,BM_MOR2019,GK_MOR2023}, with the connection to regularized empirical loss minimization now well understood~\cite{SMPK_NIPS2015,DN_AS2021}.

In our previous work~\cite{IKNU_ORL2026}, we studied DRO with a $q$-norm ambiguity set ($q$-DRO) for probability smoothing.
The resulting estimator was shown to satisfy Positivity, Symmetry (for $q \in [1, \infty]$) and Order Preservation (for $q \in (1,\infty)$), while relaxing Ratio Preservation.
The analysis proceeded through the KKT conditions of a conic reformulation, and identified the specific regularization structure induced by the $q$-norm ambiguity set as an empirical cross-entropy minimization with a $q^*$-norm penalty on the log-probability vector.

Among the axioms preserved by $q$-DRO, Symmetry, which requires that categories with identical empirical frequencies receive identical probabilities, deserves closer scrutiny.
From an axiomatic standpoint, Symmetry is regarded as a natural fairness condition.
When read in terms of its practical content, however, it imposes a strong constraint: since all unobserved categories share the empirical frequency $\hat{p}_j = 0$,
Symmetry forces any estimator satisfying it to assign the same probability to every unobserved category.
Laplace smoothing obeys this prescription by construction, and $q$-DRO, despite inducing a different regularization structure, inherits the same symmetric treatment.

In many applications, however, unobserved events are not genuinely interchangeable.
In natural language processing, an unseen $n$-gram composed of semantically related words should be considered more likely than an unseen $n$-gram of unrelated words.
Optimal transport provides a principled way to encode such structural information into the ambiguity set through a cost matrix $C = (c_{ij})$ specifying pairwise dissimilarity between categories.
The resulting DRO formulation (OT-DRO) breaks Symmetry by design,
but respects the geometric structure of the categories.

\paragraph{Contributions.}
Building on these motivations, our contributions are threefold.
\begin{enumerate}
  \item \textbf{A unifying maximum entropy principle.}
  We show that, for probability smoothing, distributionally robust optimization over any compact convex ambiguity set that contains a strictly positive distribution is equivalent to maximum entropy estimation over that ambiguity set. This equivalence follows from Sion's minimax theorem~\cite{S_PJM1958} and can be viewed as a specialization of the general minimax--entropy duality developed by Gr\"unwald and Dawid~\cite{GD_AS2004} to the cross-entropy loss on the probability simplex. While the underlying duality is classical, our contribution is to reveal its structural implication for distributionally robust probability smoothing, providing a unified view of a broad class of DRO models.

  \item \textbf{OT-DRO for probability smoothing.}
  As an important instance of the above framework, we formulate an optimal transport distributionally robust optimization (OT-DRO) model for probability smoothing under the cross-entropy loss. Exploiting the duality for optimal transport ambiguity sets, we derive an equivalent tractable reformulation as an empirical cross-entropy minimization problem with a Lipschitz-type penalty on the log-probability vector.

  \item \textbf{A unified characterization of threshold phenomena.}
  As a consequence of the above maximum entropy principle, we show that both norm-based DRO and OT-DRO exhibit the same threshold behavior: the optimal estimator coincides with the uniform distribution if and only if the uniform distribution belongs to the ambiguity set. In particular, for OT-DRO, we derive the explicit threshold $d_{\mathrm{OT}}(\vt{1}_n/n,\hat{\vt{p}})$, which parallels the threshold $\|\vt{1}_n/n-\hat{\vt{p}}\|_q$ for $q$-DRO. Consequently, in both formulations, the threshold is the distance from the empirical distribution to the uniform distribution measured by the discrepancy defining the ambiguity set.
\end{enumerate}

The paper is organized as follows. Section~\ref{sec:pre} fixes notation, recalls the DRO framework for probability smoothing, and summarizes the results of~\cite{IKNU_ORL2026} for $q$-DRO. Section~\ref{sec:otdro} develops the OT-DRO formulation and establishes its threshold characterization. Section~\ref{sec:unified} presents the unifying maximum entropy interpretation.
Section~\ref{sec:experiments} reports numerical experiments,
and Section~\ref{sec:conclusion} concludes the paper.

\section{Preliminaries}
\label{sec:pre}
\subsection{Axiomatic Analysis and DRO Framework for Probability Smoothing}

Let $N = \{1, 2, \ldots, n\}$ be a set of categories, and let $\Delta^n = \{\vt{p} \in \mathbb{R}^n \mid \sum_{j=1}^n p_j = 1,\, p_{j}\ge 0 \}$ denote the probability simplex. Given an empirical distribution $\hat{\vt{p}} \in \Delta^n$, the smoothing problem seeks an estimator $\vt{x}$ in the relative interior of $\Delta^n$, which we denote by $\operatorname{ri}(\Delta^{n})=\{\vt{x} \in \Delta^{n} \mid x_{j}>0 \text{ for all $j \in N$}\}$.
The motivating situation is the zero-frequency problem, in which $\hat{\vt{p}}$ has at least one zero component and is therefore unusable as an estimator; the theory developed below, however, requires no such assumption and applies to an arbitrary $\hat{\vt{p}} \in \Delta^{n}$.

Laplace smoothing estimates $x_j=(\hat{p}_{j}+a)/(1+na)$ with a smoothing parameter $a>0$, which amounts to adding the same pseudocount to every category.
Sakai~\cite{S_MSS2025} characterized Laplace smoothing by four axioms: Positivity ($x_j > 0$ for all $j \in N$), Symmetry ($\hat{p}_i = \hat{p}_j \Rightarrow x_i = x_j$), Order Preservation ($\hat{p}_i < \hat{p}_j \Rightarrow x_i < x_j$), and Ratio Preservation.
The previous work~\cite{IKNU_ORL2026} showed that the $q$-DRO estimator (defined below) relaxes Ratio Preservation while retaining the other three axioms for $q \in (1,\infty)$.

We adopt the cross-entropy loss $\mathcal{L}(\vt{x}, \vt{p}) = \sum_{j=1}^n p_j(-\log x_j)$, which is a standard choice for probability estimation through its connection to maximum likelihood estimation and the Kullback--Leibler divergence~\cite{CT_EIT2006}.
We consider the DRO problem
\begin{align}
\label{eq:DRO}
\min_{\vt{x} \in \Delta^n} \max_{\vt{p} \in \mathcal{U}(\hat{\vt{p}}, \varepsilon)} \sum_{j=1}^n p_j (-\log x_j),
\end{align}
where $\mathcal{U}(\hat{\vt{p}}, \varepsilon) \subseteq \Delta^n$ is an ambiguity set and $\varepsilon > 0$ is the robustness radius.
We assume the following condition on the ambiguity set.

\begin{assumption}
\label{asm:ambiguity}
The ambiguity set $\mathcal{U}(\hat{\vt{p}},\varepsilon)$ is a nonempty compact convex subset of $\Delta^{n}$ such that $\mathcal{U}(\hat{\vt{p}},\varepsilon) \cap \operatorname{ri}(\Delta^{n})\neq \emptyset$.
\end{assumption}

This mild condition is satisfied by both ambiguity sets considered in this paper, as we verify in Remark~\ref{rem:qdro-set} and Lemma~\ref{lem:otdro-set}.
Under Assumption~\ref{asm:ambiguity}, the DRO problem~\eqref{eq:DRO} admits an optimal solution, which we prove in Theorem~\ref{thm:unified}.
Here and below, $\vt{1}_{n}$ denotes the $n$-dimensional vector of all ones, so that $\vt{1}_{n}/n$ is the uniform distribution.

Throughout the paper, we regard $-\log x_j$ as an extended-real-valued function on $[0,1]$ and adopt the conventions $-\log 0 = +\infty$ and $0 \cdot (-\log 0) = 0$;
the loss $\mathcal{L}(\vt{x}, \vt{p})$ thus takes the value $+\infty$ whenever $x_j = 0$ and $p_j > 0$ for some $j \in N$.
Note also that $x_{j} \le 1$ implies $-\log x_{j} \ge 0$, so every term of $\mathcal{L}(\vt{x},\vt{p})$ is nonnegative.
The first convention structurally enforces the goal of smoothing.
We also write
\begin{align*}
  H(\vt{p}) = \sum_{j=1}^{n}p_{j}(-\log p_{j}) \qquad (\vt{p} \in \Delta^{n})
\end{align*}
for the Shannon entropy; by the convention $0\cdot(-\log 0)=0$, $H(\vt{p}) = \mathcal{L}(\vt{p},\vt{p})$ is well-defined and continuous on $\Delta^{n}$.

\subsection{Review of $q$-DRO Model}
Throughout this subsection, we assume $\varepsilon>0$ and $q \in [1,\infty]$.
The $q$-norm ambiguity set, introduced in~\cite{IKNU_ORL2026}, is
\begin{align*}
\mathcal{U}_q(\hat{\vt{p}}, \varepsilon) = \left\{\vt{p} \in \Delta^n \mid \|\vt{p} - \hat{\vt{p}}\|_q \leq \varepsilon\right\},
\end{align*}

\begin{remark}
\label{rem:qdro-set}
The set $\mathcal{U}_q(\hat{\vt{p}}, \varepsilon)$ satisfies Assumption~\ref{asm:ambiguity}.
Indeed, it contains $\hat{\vt{p}}$ and is the intersection of the compact convex set $\Delta^{n}$ with a closed norm ball, hence compact and convex.
Moreover, the point $\vt{p}^{(t)} = (1-t)\hat{\vt{p}} + t\,\vt{1}_{n}/n$ has strictly positive components for every $t \in (0,1]$ and satisfies $\|\vt{p}^{(t)} - \hat{\vt{p}}\|_{q} = t\,\|\vt{1}_{n}/n - \hat{\vt{p}}\|_{q} \le \varepsilon$ for sufficiently small $t>0$, so that $\vt{p}^{(t)} \in \mathcal{U}_q(\hat{\vt{p}}, \varepsilon)\cap \operatorname{ri}(\Delta^{n})$ for such $t$.
\end{remark}

Introducing a perturbation vector $\vt{e} = (e_j)_{j \in N}$ to represent the deviation from the empirical distribution ($p_{j}=\hat{p}_{j}+e_{j}$),
we obtain the following equivalent reformulation of the ambiguity set $\mathcal{U}_q(\hat{\vt{p}}, \varepsilon)$, where $\sum_{j=1}^{n} e_j = 0$ corresponds to the normalization $\sum_{j=1}^{n} p_j = 1$ and $\hat{p}_j + e_j \ge 0$ to the nonnegativity $p_j \ge 0$:
\begin{align*}
  \hat{p}_{j}+e_{j} & \ge 0 \text{ for $j \in N$},\quad (\lambda_{j} \ge 0)\\
  \sum_{j=1}^{n}e_{j}&=0, \quad (\beta \in \bb{R})\\
  \|\vt{e}\|_{q} &\le \varepsilon.
\end{align*}
Here, the variables shown in parentheses to the right of the first two constraints denote the Lagrange multipliers associated with them, which are used in the dual reformulation of the inner problem.

The $q$-DRO model is problem~\eqref{eq:DRO} with $\mathcal{U} = \mathcal{U}_q(\hat{\vt{p}}, \varepsilon)$:
\begin{align}
\label{eq:qdro}
\text{$q$-DRO:}\qquad
\min_{\vt{x} \in \Delta^n} \max_{\vt{p} \in \mathcal{U}_q(\hat{\vt{p}}, \varepsilon)} \sum_{j=1}^n p_j(-\log x_j).
\end{align}

The $q$-DRO problem~\eqref{eq:qdro} was reformulated in~\cite{IKNU_ORL2026} via Lagrangian duality, as summarized in the following theorem.
\begin{theorem}[$q$-DRO reformulation~\cite{IKNU_ORL2026}]
\label{thm:qdro-equiv}
A solution $\vt{x}^*$ is optimal to the $q$-DRO problem~\eqref{eq:qdro} if and only if there exists $\beta^* \in \mathbb{R}$ such that $(\vt{x}^*, \beta^*)$ is an optimal solution of
\begin{align}
\label{eq:qdro-simplified}
\min_{\vt{x} \in \Delta^n,\, \beta \in \bb{R}} \sum_{j=1}^n \hat{p}_j (-\log x_j) + \varepsilon \|-\log(\vt{x}) - \beta \vt{1}_{n}\|_{q^*},
\end{align}
where $q^*$ is the H\"older conjugate of $q$, i.e., $1/q+1/q^{*}=1$. Moreover, the two problems share the same optimal value.
\end{theorem}

The reformulation~\eqref{eq:qdro-simplified} exhibits $q$-DRO as an empirical cross-entropy minimization with a $q^*$-norm penalty on the deviation of the log-probability vector $-\log(\vt{x})$ from a common baseline $\beta$.
We note that the penalty term is a nonconvex function of $\vt{x}$; nevertheless, since it is the composition of the finite convex function $\|\cdot\|_{q^{*}}$ with the continuously differentiable map $(\vt{x},\beta) \mapsto -\log(\vt{x})-\beta\vt{1}_{n}$, the subdifferential chain rule applies, its qualification condition holding automatically since $\|\cdot\|_{q^{*}}$ is finite valued (see~\cite[Thm.~10.6]{RW_VA1998}), and the resulting KKT conditions serve as necessary conditions for optimality.

A threshold condition on the robustness radius is also established.
\begin{theorem}[$q$-DRO threshold]
\label{thm:qdro-threshold}
The uniform distribution $\vt{x} = \vt{1}_n/n$ is an optimal solution to $q$-DRO if and only if $\|\vt{1}_n/n - \hat{\vt{p}}\|_q \le \varepsilon$.
\end{theorem}

The proof is given in Appendix~\ref{app:threshold_qdro}. This threshold corresponds to the radius at which the uniform distribution belongs to the ambiguity set, a perspective that will be made systematic in Section~\ref{sec:unified}.

\section{The OT-DRO Formulation}
\label{sec:otdro}
In this section, we apply the duality for optimal transport ambiguity sets to the probability smoothing problem. The duality itself is standard, and is usually stated for Wasserstein ambiguity sets~\cite{MK_MM2018,BM_MOR2019,GK_MOR2023}; what we wish to highlight is its specialization to the cross-entropy loss on the probability simplex, which yields an empirical cross-entropy minimization with a Lipschitz-type penalty on the log-probability vector and a threshold condition paralleling that of $q$-DRO.

The goal of this section is to formulate the OT-DRO model and to derive a tractable convex reformulation that can be solved by standard optimization software. We proceed in three steps.
First, we express the inner maximization as a linear programming problem over transport plans (Section~\ref{subsec:inner}).
Second, we take its linear-programming dual and combine it with the outer minimization, obtaining an equivalent reformulation (Section~\ref{subsec:dual}).
Finally, we analyze the resulting penalty structure and derive a threshold condition on the robustness radius (Section~\ref{subsec:properties}).

\subsection{The Inner Problem}
\label{subsec:inner}
The OT-DRO model is problem~\eqref{eq:DRO} with $\mathcal{U} = \mathcal{U}_{\mathrm{OT}}(\hat{\vt{p}}, \varepsilon)$:
\begin{align}
\label{eq:otdro}
\text{OT-DRO:}\qquad
\min_{\vt{x} \in \Delta^n} \max_{\vt{p} \in \mathcal{U}_{\mathrm{OT}}(\hat{\vt{p}}, \varepsilon)} \sum_{j=1}^n p_j(-\log x_j),
\end{align}
where the optimal transport ambiguity set is $\mathcal{U}_{\mathrm{OT}}(\hat{\vt{p}}, \varepsilon) = \{\vt{p} \in \Delta^n \mid d_{\mathrm{OT}}(\vt{p}, \hat{\vt{p}}) \leq \varepsilon\}$.

For $\vt{x} \in \Delta^{n}$, the inner maximization in~\eqref{eq:otdro} takes the form
\begin{align}
\label{eq:otdro-inner-p}
\max_{\vt{p} \in \Delta^n}~ \sum_{j=1}^n p_j(-\log x_j) \quad \st \quad d_{\mathrm{OT}}(\vt{p}, \hat{\vt{p}}) \leq \varepsilon.
\end{align}
The optimal transport discrepancy $d_{\mathrm{OT}}(\vt{p}, \hat{\vt{p}})$ measures the minimum cost of transporting the probability mass of $\hat{\vt{p}}$ to $\vt{p}$, where the transport cost is specified by a cost matrix $C = (c_{ij}) \in \mathbb{R}^{n \times n}$.
As introduced in Section~\ref{sec:intro}, the entry $c_{ij}$ represents the cost of moving a unit of mass from category $i$ to category $j$, and thereby encodes the geometric structure of the categories: a small $c_{ij}$ means that categories $i$ and $j$ are close.
We assume that $c_{ij} \geq 0$ and $c_{ii} = 0$ for all $i,j \in N$.
Note that we do not assume $c_{ij} = c_{ji}$, nor the triangle inequality, so $d_{\mathrm{OT}}$ is in general not a metric; for this reason, we use the term \emph{optimal transport discrepancy} rather than \emph{Wasserstein distance}.
When $C$ is a metric on $N$, the discrepancy $d_{\mathrm{OT}}$ reduces to the $1$-Wasserstein distance and OT-DRO becomes the Wasserstein DRO model studied in~\cite{MK_MM2018,BM_MOR2019,GK_MOR2023}.

Since the constraint $d_{\mathrm{OT}}(\vt{p}, \hat{\vt{p}}) \leq \varepsilon$ is defined implicitly through an optimal transport problem, we obtain an explicit formulation by lifting the problem: we introduce the transport plan $\vt{y} = (y_{ij}) \in \mathbb{R}^{n \times n}_{\geq 0}$ as a decision variable.
A variable $\vt{y}$ is a transport plan from $\hat{\vt{p}}$ to $\vt{p}$ if and only if $\sum_{j=1}^{n} y_{ij} = \hat{p}_i$ for all $i \in N$ and $\sum_{i=1}^{n} y_{ij} = p_j$ for all $j \in N$.
The optimal transport discrepancy is the minimum of $\sum_{i=1}^{n}\sum_{j=1}^{n} c_{ij} y_{ij}$ over all such transport plans; in particular, $d_{\mathrm{OT}}(\vt{p}, \hat{\vt{p}}) \leq \varepsilon$ if and only if there exists a transport plan $\vt{y}$ from $\hat{\vt{p}}$ to $\vt{p}$ with $\sum_{i=1}^{n}\sum_{j=1}^{n} c_{ij} y_{ij} \leq \varepsilon$.

This description yields the following properties of the ambiguity set.
\begin{lemma}
\label{lem:otdro-set}
Let $\varepsilon>0$, and assume $c_{ij} \ge 0$ and $c_{ii}=0$ for all $i,j \in N$.
Then $\mathcal{U}_{\mathrm{OT}}(\hat{\vt{p}}, \varepsilon)$ satisfies Assumption~\ref{asm:ambiguity}.
\end{lemma}
The proof is given in Appendix~\ref{app:ambiguity}.

Substituting $p_j = \sum_{i=1}^{n} y_{ij}$ into the objective of~\eqref{eq:otdro-inner-p} and eliminating $\vt{p}$, we obtain the problem
\begin{subequations}
\label{eq:otdro-inner}
\begin{align}
  \max_{\vt{y}}~& \sum_{i=1}^{n}\sum_{j=1}^{n} (-\log x_j) y_{ij} \label{eq:otdro-inner-obj}\\
\st~& \sum_{j=1}^n y_{ij} = \hat{p}_i \text{ for $i \in N$}, \label{eq:otdro-inner-c1}\\
& \sum_{i=1}^{n}\sum_{j=1}^{n} c_{ij} y_{ij} \leq \varepsilon, \label{eq:otdro-inner-c2}\\
& y_{ij} \ge 0 \text{ for $i,j \in N$.}
\end{align}
\end{subequations}

The marginal constraints $\sum_{i=1}^{n} y_{ij} = p_j$ have been absorbed into the objective via $p_j = \sum_{i=1}^{n} y_{ij}$, so they no longer appear explicitly; the nonnegativity $\vt{p} \geq \vt{0}$ is automatically satisfied since $y_{ij} \geq 0$, and the simplex condition $\sum_{j=1}^{n} p_j = 1$ follows from the constraint~\eqref{eq:otdro-inner-c1} via $\sum_{j=1}^{n} p_j = \sum_{i=1}^{n}\sum_{j=1}^{n} y_{ij} = \sum_{i=1}^{n} \hat{p}_{i} = 1$.
When $x_{j}>0$ for all $j \in N$, every coefficient $-\log x_{j}$ is finite and~\eqref{eq:otdro-inner} is a linear programming (LP) problem; when $x_{k}=0$ for some $k \in N$, its objective takes values in $[0,+\infty]$ under the conventions of Section~\ref{sec:pre}.
Since the substitution $p_{j} = \sum_{i=1}^{n}y_{ij}$ does not change the objective value,
problems~\eqref{eq:otdro-inner-p} and~\eqref{eq:otdro-inner} clearly have the same optimal
value for every $\vt{x} \in \Delta^{n}$.

\subsection{Dual Reformulation}
\label{subsec:dual}
We first derive the dual of~\eqref{eq:otdro-inner} for $\vt{x} \in \Delta^{n}$ with $x_{j}>0$ for all $j \in N$, so that~\eqref{eq:otdro-inner} is an LP.
Introducing Lagrange multipliers $\mu_i \in \mathbb{R}$ for the equality constraints~\eqref{eq:otdro-inner-c1} and $\alpha \geq 0$ for the budget constraint~\eqref{eq:otdro-inner-c2}, we obtain the Lagrangian
\begin{align*}
L(\vt{y}, \vt{\mu}, \alpha)
&= \sum_{i=1}^{n}\sum_{j=1}^{n}(-\log x_j) y_{ij} + \sum_{i=1}^{n} \mu_i\left(\hat{p}_i - \sum_{j=1}^{n} y_{ij}\right) + \alpha\left(\varepsilon - \sum_{i=1}^{n}\sum_{j=1}^{n} c_{ij} y_{ij}\right)\\
&=\sum_{i=1}^{n}\sum_{j=1}^{n} (-\log x_j - \mu_i - c_{ij}\alpha) y_{ij} + \sum_{i=1}^{n} \hat{p}_i \mu_i + \varepsilon \alpha.
\end{align*}

The supremum over $\vt{y} \geq \vt{0}$ is finite if and only if $(-\log x_j) - \mu_i - c_{ij}\alpha \leq 0$ for all $i, j \in N$; otherwise, choosing $y_{ij} \to \infty$ for such a pair $(i,j)$ drives the Lagrangian to $+\infty$.
Under this condition, the supremum is attained at $\vt{y} = \vt{0}$.
Hence, the dual problem is
\begin{subequations}
\label{eq:otdro-inner-dual}
\begin{align}
  \min_{\alpha \geq 0,\, \vt{\mu}\in\bb{R}^{n}}~ & \sum_{i=1}^n \hat{p}_i \mu_i +\varepsilon \alpha \label{eq:otdro-inner-dual_obj}\\
\st \quad & -\log x_j - c_{ij} \alpha \leq \mu_i \text{ for $i, j \in N$}. \label{eq:otdro-inner-dual_con}
\end{align}
\end{subequations}

Note that problem~\eqref{eq:otdro-inner-dual} is also defined when $x_{k}=0$ for some
$k \in N$.
In this case, the constraint~\eqref{eq:otdro-inner-dual_con} for the pair $(i,k)$ becomes $+\infty \le \mu_{i}$.
No real number $\mu_{i}$ satisfies this constraint, so~\eqref{eq:otdro-inner-dual} is
infeasible.
We adopt the convention that the infimum over the empty set is $+\infty$, so that the optimal
value of~\eqref{eq:otdro-inner-dual} is $+\infty$ in this case.
The following lemma shows that the primal problem~\eqref{eq:otdro-inner} and the dual problem~\eqref{eq:otdro-inner-dual} have the same optimal value for every $\vt{x} \in \Delta^{n}$, including the points on the relative boundary of $\Delta^{n}$.

\begin{lemma}
\label{lem:otdro-inner-dual}
Let $\vt{x} \in \Delta^{n}$.
Then problems~\eqref{eq:otdro-inner} and~\eqref{eq:otdro-inner-dual} have the same optimal value.
Moreover, the following statements hold.
\begin{itemize}
  \item[(i)] If $x_{j}>0$ for all $j \in N$, then the optimal value is finite, and both problems have optimal solutions.
  \item[(ii)] If $x_{k}=0$ for some $k \in N$, then the optimal value is $+\infty$.
\end{itemize}
\end{lemma}

\begin{proof}
(i) Suppose that $x_{j}>0$ for all $j \in N$.
By the above derivation, \eqref{eq:otdro-inner-dual} is the LP dual of~\eqref{eq:otdro-inner}.
Problem~\eqref{eq:otdro-inner} is feasible, because the identity coupling ($y_{ii} = \hat{p}_{i}$ and $y_{ij} = 0$ for $i \ne j$) satisfies its constraints.
In addition, its feasible region is bounded, because $0 \le y_{ij} \le \hat{p}_{i}$.
Therefore, by the LP duality theorem, both \eqref{eq:otdro-inner} and~\eqref{eq:otdro-inner-dual} have optimal solutions, and their optimal values are finite and equal.

(ii) Suppose that $x_{k}=0$ for some $k \in N$.
In this case, \eqref{eq:otdro-inner} is not an LP, so we compute the two optimal values directly.
By Lemma~\ref{lem:otdro-set}, there exists $\tilde{\vt{p}} \in \mathcal{U}_{\mathrm{OT}}(\hat{\vt{p}},\varepsilon)\cap \operatorname{ri}(\Delta^{n})$; in particular, $\tilde{p}_{j}>0$ for all $j \in N$.
Since $\tilde{\vt{p}} \in \mathcal{U}_{\mathrm{OT}}(\hat{\vt{p}},\varepsilon)$, there exists
a transport plan $\tilde{\vt{y}}$ from $\hat{\vt{p}}$ to $\tilde{\vt{p}}$ such that
$\sum_{i,j}c_{ij}\tilde{y}_{ij} \le \varepsilon$.
This $\tilde{\vt{y}}$ satisfies the constraints of~\eqref{eq:otdro-inner}, so it is a
feasible solution of~\eqref{eq:otdro-inner}.
Since $\tilde{\vt{y}}$ is a transport plan to $\tilde{\vt{p}}$, we have
$\sum_{i}\tilde{y}_{ik} = \tilde{p}_{k}>0$, so $\tilde{y}_{ik}>0$ for some $i \in N$.
Since $-\log x_{k} = +\infty$ and all the other terms of the objective function
of~\eqref{eq:otdro-inner} are nonnegative, the objective value at $\tilde{\vt{y}}$ is
$+\infty$.
Therefore, the optimal value of~\eqref{eq:otdro-inner} is $+\infty$.
On the other hand, as shown above, \eqref{eq:otdro-inner-dual} is infeasible, and its
optimal value is $+\infty$.\qed
\end{proof}

Replacing the inner problem with its dual~\eqref{eq:otdro-inner-dual} and combining with the outer minimization yields a single minimization problem.
\begin{subequations}
\label{eq:otdro-reform}
\begin{align}
  \min_{(\vt{x}, \alpha, \vt{\mu})}~ & \sum_{i=1}^n \hat{p}_i \mu_i +\varepsilon \alpha \label{eq:otdro-reform-obj}\\
\st \quad & \sum_{j=1}^n x_j = 1, \label{eq:otdro-reform-con1}\\
& -\log x_j - c_{ij} \alpha \leq \mu_i \text{ for $i, j \in N$}, \label{eq:otdro-reform-con2}\\
& \alpha \geq 0, \quad x_j \geq 0 \text{ for $j \in N$} \label{eq:otdro-reform-con3}.
\end{align}
\end{subequations}

Before stating the relationship between the OT-DRO model~\eqref{eq:otdro} and its reformulation~\eqref{eq:otdro-reform}, we show the positivity of the estimator, which holds not only at optimality but at every feasible solution of~\eqref{eq:otdro-reform}.

\begin{proposition}
\label{prop:otdro-pos}
Every feasible solution $(\vt{x}, \alpha, \vt{\mu})$ of~\eqref{eq:otdro-reform} satisfies $x_j > 0$ for all $j \in N$.
In particular, so does every optimal solution.
\end{proposition}

\begin{proof}
Let $(\vt{x}, \alpha, \vt{\mu})$ be feasible and fix $i, j \in N$.
Since $\alpha \ge 0$ and $c_{ij} \ge 0$, the constraint~\eqref{eq:otdro-reform-con2} yields $-\log x_{j} \le \mu_{i}+c_{ij}\alpha$.
The right-hand side is finite, so $-\log x_{j} < +\infty$, which means $x_{j}>0$ under the convention $-\log 0 = +\infty$.
\qed
\end{proof}

We are now in a position to state the relationship between the OT-DRO model~\eqref{eq:otdro} and its reformulation~\eqref{eq:otdro-reform}.

\begin{theorem}[OT-DRO reformulation]
\label{thm:otdro-equiv}
Suppose $\varepsilon>0$.
Then both the OT-DRO model~\eqref{eq:otdro} and the problem~\eqref{eq:otdro-reform} have optimal solutions, and the two problems share the same optimal value.
Moreover, a solution $\vt{x}^*$ is optimal to~\eqref{eq:otdro} if and only if there exist $\alpha^* \geq 0$ and $\vt{\mu}^* \in \mathbb{R}^n$ such that $(\vt{x}^*, \alpha^*, \vt{\mu}^*)$ is an optimal solution of~\eqref{eq:otdro-reform}.
\end{theorem}

The proof is given in Appendix~\ref{app:solvable}.
It combines Lemma~\ref{lem:otdro-inner-dual} with Theorem~\ref{thm:unified}, which supplies the optimal solution $\vt{x}^{*}$ of~\eqref{eq:otdro}; the accompanying multipliers $(\alpha^{*},\vt{\mu}^{*})$ are then optimal solutions of~\eqref{eq:otdro-inner-dual} at $\vt{x} = \vt{x}^{*}$.

The formulation~\eqref{eq:otdro-reform} is a convex optimization problem: the objective is linear, and the only nonlinear constraint $-\log x_j - c_{ij}\alpha \leq \mu_i$ is convex, since $-\log{x_j}$ is convex in $x_j$.
The logarithmic terms can be represented via exponential cone constraints, enabling efficient solution by off-the-shelf conic optimization solvers.
Moreover, the problem satisfies Slater's condition: the solution $(\vt{x},\alpha,\vt{\mu}) = (\vt{1}_n/n, 1, (\log{n}+\max_{i,j}\{c_{ij}\}+1)\vt{1}_{n})$ is feasible and satisfies every inequality constraint strictly.
Consequently, the KKT conditions are necessary and sufficient for optimality.

The role of the robustness radius $\varepsilon$ can be understood directly from~\eqref{eq:otdro-reform}.
The multiplier $\alpha$, associated with the transport budget, appears in the objective as $\varepsilon\alpha$ and in the constraints as $c_{ij}\alpha$.
When $\varepsilon$ is sufficiently large, the term $\varepsilon\alpha$ heavily penalizes positive values of $\alpha$, so the optimal solution takes $\alpha = 0$.
In this case, the cost matrix $c_{ij}$ drops out of the constraints, which reduce to $-\log x_j \leq \mu_i$ for all $i, j \in N$, independent of $i$.
Minimizing the objective $\sum_i \hat{p}_i \mu_i$ then amounts to maximizing $\min_j x_j$ over the simplex, whose solution is the uniform distribution $\vt{x} = \vt{1}_n/n$, since the minimum component is maximized when all components are equal.
Thus, as $\varepsilon$ grows, the estimator is driven toward the uniform distribution.

\subsection{Properties of OT-DRO}
\label{subsec:properties}
The constraints $-\log x_j - c_{ij}\alpha \leq \mu_i$ can be rewritten as $\mu_i \geq \max_{j \in N}\{-\log x_j - c_{ij}\alpha\}$.
At any optimal solution this holds with equality for every $i$ with $\hat{p}_{i}>0$, because increasing such a $\mu_{i}$ strictly increases the objective.
For the remaining indices the value of $\mu_{i}$ does not affect the objective, since $\hat{p}_{i}=0$, so we may assume the equality for all $i \in N$ without loss of generality.
Moreover,
by Proposition~\ref{prop:otdro-pos}, every feasible solution of~\eqref{eq:otdro-reform} satisfies $x_{j}>0$ for all $j \in N$, so all the logarithms below are finite,
and $\max_{j \in N}\{-\log x_j - c_{ij}\alpha\}$ is rewritten as
\begin{align}
  \max_{j \in N}\{-\log x_j - c_{ij}\alpha\}
  &=\max_{j \in N}\{-\log x_{i} + \log x_{i}-\log x_j - c_{ij}\alpha\}
  =-\log x_{i}+\max_{j \in N}\left\{ \log\left(\frac{x_{i}}{x_{j}}\right) - c_{ij}\alpha\right\}. \label{eq:penalty}
\end{align}

As noted above, the $\vt{x}$-component of every feasible solution of~\eqref{eq:otdro-reform} belongs to $\operatorname{ri}(\Delta^{n})$.
Hence, substituting~\eqref{eq:penalty} into~\eqref{eq:otdro-reform-obj}, we obtain the following equivalent formulation:
\begin{align*}
\min_{\vt{x} \in \operatorname{ri}(\Delta^n),\, \alpha \geq 0}~ \sum_{i=1}^n \hat{p}_{i}(-\log x_{i})+\sum_{i=1}^{n}\hat{p}_{i}\max_{j \in N}\left\{ \log\left(\frac{x_{i}}{x_{j}}\right) - c_{ij}\alpha\right\} +\varepsilon \alpha.
\end{align*}

The first term is the empirical cross-entropy loss, and the second term acts as a Lipschitz-type penalty on the log-probability vector.
To see this, note that the term $\max_{j \in N}\left\{\log(x_i/x_j) - c_{ij}\alpha\right\}$ evaluates the maximum log-ratio $\log(x_{i}/x_{j})$ across all categories $j \in N$, penalized in proportion to the cost $c_{ij}\alpha$.
The argument at $j=i$ equals zero (since $c_{ii}=0$), so the maximum is always nonnegative, and it becomes positive when some category $j$ with small $c_{ij}$ has $x_j < x_i$, i.e., when the estimated probability drops sharply between neighboring categories.
The penalty thus restricts the variation of estimated probabilities across categories in proportion to their pairwise distances: categories with small $c_{ij}$ are coupled through the penalty, encouraging similar estimated probabilities, while categories with large $c_{ij}$ are decoupled and can differ more freely.
The variable $\alpha$ controls the radius of this coupling: small $\alpha$ allows widespread coupling between distant categories, while large $\alpha$ localizes the coupling to nearby categories, with the penalty $\varepsilon\alpha$ on $\alpha$ balancing these effects.
This stands in clear contrast to the $q^*$-norm penalty of $q$-DRO in~\eqref{eq:qdro-simplified}, which treats all categories symmetrically.

We now characterize the condition under which OT-DRO returns the uniform distribution.
The proof of Theorem~\ref{thm:otdro-threshold} is also given in Appendix~\ref{app:threshold_otdro}.

\begin{theorem}[OT-DRO threshold]
\label{thm:otdro-threshold}
The uniform distribution $\vt{x} = \vt{1}_n/n$ is an optimal solution to OT-DRO if and only if $d_{\mathrm{OT}}(\vt{1}_n/n, \hat{\vt{p}})\le \varepsilon$.
\end{theorem}

\begin{remark}\label{rem:parallel}
The threshold $d_{\mathrm{OT}}(\vt{1}_n/n, \hat{\vt{p}})$ for OT-DRO and the threshold $\|\vt{1}_n/n - \hat{\vt{p}}\|_q$ for $q$-DRO (Theorem~\ref{thm:qdro-threshold}) share a common structure: in each case, the threshold is the distance from the empirical distribution to the uniform distribution, measured by the discrepancy that defines the ambiguity set. This parallel, obtained through separate analyses of two formulations with contrasting symmetry properties, suggests a common underlying principle, which we identify in the next section.
\end{remark}

\section{A Unifying Perspective via Maximum Entropy}
\label{sec:unified}
The two formulations studied above originate from contrasting motivations. The $q$-norm ambiguity set treats categories symmetrically and the resulting estimator satisfies Symmetry; the optimal transport ambiguity set breaks this symmetry through the cost matrix to enable structure-respecting smoothing. On the surface, these are opposing design choices. Yet in both cases the threshold is the distance from $\hat{\vt{p}}$ to $\vt{1}_n/n$ measured by the corresponding discrepancy.

We now show that this parallel is a direct consequence of a reformulation of the DRO problem~\eqref{eq:DRO} that applies to any ambiguity set satisfying Assumption~\ref{asm:ambiguity}, and which provides a unified explanation of these structural properties.

\subsection{Sion's Minimax Theorem and the Reformulation}
\begin{theorem}
\label{thm:unified}
Let $\mathcal{U} = \mathcal{U}(\hat{\vt{p}}, \varepsilon)$ satisfy Assumption~\ref{asm:ambiguity}.
Then the DRO problem~\eqref{eq:DRO} has a unique optimal solution $\vt{x}^*$, and $\vt{x}^* = \vt{p}^*$, where
\begin{align*}
\vt{p}^* = \underset{\vt{p} \in \mathcal{U}}{\mathop{{\rm argmax}}}\, H(\vt{p})
\end{align*}
is the unique maximizer of the Shannon entropy $H$ over $\mathcal{U}$.
Moreover, the optimal value of~\eqref{eq:DRO} equals $H(\vt{p}^*)$.
In particular, $x^*_{j} >0$ for all $j \in N$.
\end{theorem}

\begin{proof}
Throughout the proof, let $\Delta_{\delta}^{n} \coloneqq \{ \vt{x} \in \mathbb{R}^{n} \mid \vt{1}_n^{\top} \vt{x} = 1, \vt{x} \geq \delta \vt{1}_n \}$, $\varphi(\delta) \coloneqq \inf_{\vt{x} \in \Delta_{\delta}^n} \sup_{\vt{p} \in \mathcal{U}} \mathcal{L}(\vt{x},\vt{p})$, and $\psi(\delta) \coloneqq \sup_{\vt{p} \in \mathcal{U}} \inf_{\vt{x} \in \Delta_{\delta}^n} \mathcal{L}(\vt{x},\vt{p})$ for all $\delta \in [0, \textstyle\frac{1}{n}]$.

We first show that
\begin{align} 
\lim_{\delta \to +0} \varphi(\delta) = \varphi(0), \qquad \lim_{\delta \to +0} \psi(\delta) = \psi(0).
\label{eq:minmax_lim}
\end{align}
Let $\vt{x} \in \Delta^n$ and $\delta \in (0,\frac{1}{n})$.
We define $g(\vt{x}) \coloneqq \sup_{\vt{p} \in \mathcal{U}} \mathcal{L}(\vt{x},\vt{p})$ and $\vt{z} \coloneqq (1 - \delta n) \vt{x} + \delta \vt{1}_n$.
From $-\log z_j \leq -\log((1-\delta n)x_{j}) = -\log x_j - \log(1-\delta n)$ for all $j \in N$,
we get $\mathcal{L}(\vt{z}, \vt{p}) \leq \sum_{j=1}^n p_{j} ( - \log x_{j} - \log \left( 1 - \delta n \right) ) = \mathcal{L}(\vt{x},\vt{p}) - \log ( 1 - \delta n )$ for any $\vt{p} \in \mathcal{U}$, i.e., $g(\vt{z}) = \sup_{\vt{p} \in \mathcal{U}} \mathcal{L}(\vt{z}, \vt{p}) \leq \sup_{\vt{p} \in \mathcal{U}} \mathcal{L}(\vt{x},\vt{p}) - \log ( 1 - \delta n )$.
Because $\varphi(\delta) = \inf_{\vt{x} \in \Delta_{\delta}^n} g(\vt{x}) \leq g(\vt{z})$ from $\vt{z} \in \Delta_{\delta}^n$, we have $\varphi(\delta) \leq g(\vt{z}) \leq \sup_{\vt{p} \in \mathcal{U}} \mathcal{L}(\vt{x},\vt{p}) - \log ( 1 - \delta n )$.
Then, taking the infimum of $\vt{x} \in \Delta^n$ implies $\varphi(\delta) \leq \varphi(0) - \log (1 - \delta n)$.
Since $\varphi(0) \leq \varphi(\delta)$ by $\Delta_{\delta}^n \subset \Delta^n$, we obtain $\varphi(0) \leq \varphi(\delta) \leq \varphi(0) - \log (1 - \delta n)$.
It then follows from $\delta \to +0$ that the first equality of~\eqref{eq:minmax_lim} holds.

Next, we show the second equality of~\eqref{eq:minmax_lim}.
It is well known that, for every $(\vt{x},\vt{p}) \in \Delta^n \times {\cal U}$,
\begin{align}
  \label{ineq:KL}
  \mathcal{L}(\vt{x},\vt{p})-H(\vt{p})=\sum_{j=1}^{n}p_{j}\log{\left(\frac{p_{j}}{x_{j}} \right)}=D_{\mathrm{KL}}(\vt{p}\,\|\, \vt{x})\ge0,
\end{align}
with equality if and only if $\vt{x}=\vt{p}$ (Gibbs' inequality; see~\cite{CT_EIT2006}).
Thus, $\inf_{\vt{x} \in \Delta^n} \mathcal{L}(\vt{x}, \vt{p}) \geq H(\vt{p})$ for $\vt{p} \in \mathcal{U}$.
Since $H(\vt{p}) = \mathcal{L}(\vt{p},\vt{p}) \geq \inf_{\vt{x} \in \Delta^n} \mathcal{L}(\vt{x}, \vt{p})$ for $\vt{p} \in \mathcal{U} \subset \Delta^n$, we get
\begin{align} \label{eq:inf_Hp}
\inf_{\vt{x} \in \Delta^n} \mathcal{L}(\vt{x},\vt{p}) = \min_{\vt{x} \in \Delta^n} \mathcal{L}(\vt{x},\vt{p}) = H(\vt{p}) \quad \forall \vt{p} \in \mathcal{U}.
\end{align}
Let $\delta \in (0, \frac{1}{n})$ and $\vt{p} \in \mathcal{U}$ be arbitrary, and let $\vt{x}_{\delta}(\vt{p}) \coloneqq (1 - \delta n) \vt{p} + \delta \vt{1}_n$.
Note that $\vt{x}_{\delta}(\vt{p}) \geq \delta \vt{1}_n$ and $\vt{1}_n^{\top} \vt{x}_{\delta}(\vt{p}) = (1 - \delta n) \vt{1}_n^{\top} \vt{p} + \delta \vt{1}_n^{\top} \vt{1}_n = 1$, namely, $\vt{x}_{\delta}(\vt{p}) \in \Delta_{\delta}^n$.
Hence, we obtain $\inf_{\vt{x} \in \Delta_{\delta}^n} \mathcal{L}(\vt{x},\vt{p}) \leq \mathcal{L}(\vt{x}_{\delta}(\vt{p}), \vt{p}) = H(\vt{p}) - \sum_{p_j > 0} p_j \log ( 1 -\delta n + \frac{\delta}{p_j} )$.
Moreover, using $\Delta_{\delta}^n \subset \Delta^n$, $-\log(1 - \delta n + \frac{\delta}{p_j}) \leq -\log(1 - \delta n)$, and $\vt{1}_n^{\top} \vt{p} = 1$ yields $\inf_{\vt{x} \in \Delta^{n}} \mathcal{L}(\vt{x},\vt{p}) \leq \inf_{\vt{x} \in \Delta_{\delta}^n} \mathcal{L}(\vt{x},\vt{p}) \leq H(\vt{p}) - \log (1 - \delta n)$.
Since $\vt{p} \in \mathcal{U}$ is arbitrary, taking the supremum of $\vt{p} \in \mathcal{U}$ in the above leads to $\psi(0) \leq \psi(\delta) \leq \sup_{\vt{p} \in \mathcal{U}} H(\vt{p}) - \log (1 - \delta n)$.
This result and~\eqref{eq:inf_Hp} imply $\psi(0) \leq \psi(\delta) \leq \psi(0) - \log (1 - \delta n)$.
Then, taking $\delta \to +0$ yields the second equality of~\eqref{eq:minmax_lim}.

We now prove the assertions of the theorem.
By Sion's minimax theorem~\cite{S_PJM1958}, we have $\inf_{\vt{x} \in \Delta_{\delta}^n} \sup_{\vt{p} \in \mathcal{U}} \mathcal{L}(\vt{x},\vt{p}) = \sup_{\vt{p} \in \mathcal{U}} \inf_{\vt{x} \in \Delta_{\delta}^n} \mathcal{L}(\vt{x},\vt{p})$ for all $\delta \in (0,\textstyle\frac{1}{n})$.
Taking $\delta \to +0$ in both sides of this equality and using~\eqref{eq:minmax_lim} yield
\begin{align} \label{eq:infsup}
\inf_{\vt{x} \in \Delta^n} \sup_{\vt{p} \in \mathcal{U}} \mathcal{L}(\vt{x},\vt{p}) = \sup_{\vt{p} \in \mathcal{U}} \inf_{\vt{x} \in \Delta^n} \mathcal{L}(\vt{x},\vt{p}).
\end{align}
Now, the function $H$ is continuous and strictly concave on $\mathcal{U}$, and the set $\mathcal{U}$ is nonempty, compact, and convex.
Thus, there exists a unique optimizer of $\sup_{\vt{p} \in \mathcal{U}} H(\vt{p})$.
It then follows from~\eqref{eq:inf_Hp} and~\eqref{eq:infsup} that
\begin{align} \label{eq:infsup_H}
\inf_{\vt{x} \in \Delta^n} \sup_{\vt{p} \in \mathcal{U}} \mathcal{L}(\vt{x},\vt{p}) = \sup_{\vt{p} \in \mathcal{U}} H(\vt{p}) = H(\vt{p}^{\ast}) < + \infty, \qquad \mathop{{\rm argmax}}_{\vt{p} \in \mathcal{U}} H(\vt{p}) = \{ \vt{p}^{\ast} \}.
\end{align}
Meanwhile, the function $g$ is lower semicontinuous (lsc) on $\Delta^n$.
In fact, it is shown as follows: For any $\vt{p} \in \mathcal{U}$ and any sequence $\{ \vt{w}_k \} \subset \Delta^n$ converging to $\bar{\vt{w}} \in \Delta^n$, the definition of $g$ means that $g(\vt{w}_k) \geq \mathcal{L}(\vt{w}_k, \vt{p})$ for all $k \in \mathbb{N}$. Since $\mathcal{L}(\, \cdot \,, \vt{p})$ is lsc on $\Delta^n$, we get $\liminf_{k \to \infty} g(\vt{w}_k) \geq \liminf_{k \to \infty} \mathcal{L}(\vt{w}_k, \vt{p}) \geq \mathcal{L}(\bar{\vt{w}}, \vt{p})$.
Thus, taking the supremum of $\vt{p} \in \mathcal{U}$ yields $\liminf_{k \to \infty} g(\vt{w}_k) \geq \sup_{\vt{p} \in \mathcal{U}} \mathcal{L}(\bar{\vt{w}}, \vt{p}) = g(\bar{\vt{w}})$.
Moreover, $g$ is proper convex on $\Delta^n$, since $g\ge 0$ and $g(\vt{1}_n/n)=\log n$.
These facts and the compactness of $\Delta^n$ guarantee that $\inf_{\vt{x} \in \Delta^n} g(\vt{x})$ has an optimizer $\bar{\vt{x}} \in \Delta^n$, i.e.,
\begin{align} \label{eq:infsup_minsup}
\inf_{\vt{x} \in \Delta^n} \sup_{\vt{p} \in \mathcal{U}} \mathcal{L}(\vt{x},\vt{p}) = \inf_{\vt{x} \in \Delta^n} g(\vt{x}) = \min_{\vt{x} \in \Delta^n} g(\vt{x}) = \min_{\vt{x} \in \Delta^n} \sup_{\vt{p} \in \mathcal{U}} \mathcal{L}(\vt{x},\vt{p}).
\end{align}
Note that $\sup_{\vt{p} \in \mathcal{U}} \mathcal{L}(\vt{x}, \vt{p})$ is solvable for $\vt{x} \in \operatorname{ri}(\Delta^{n})$, namely, $g(\vt{x}) = \sup_{\vt{p} \in \mathcal{U}} \mathcal{L}(\vt{x}, \vt{p}) = \max_{\vt{p} \in \mathcal{U}} \mathcal{L}(\vt{x}, \vt{p})$ for $\vt{x} \in \operatorname{ri}(\Delta^{n})$.
For $\vt{x} \in \Delta^n \setminus \operatorname{ri}(\Delta^{n})$, the same result indeed holds because
\begin{align} \label{cond:x}
\vt{x} \in \Delta^n \setminus \operatorname{ri}(\Delta^{n}) \quad \Longrightarrow \quad \exists \widehat{\vt{q}} \in \mathcal{U}, ~~ g(\vt{x}) = \mathcal{L}(\vt{x}, \widehat{\vt{q}}) = + \infty.
\end{align}
Now, we verify~\eqref{cond:x}.
Let $k \in N$ satisfy $x_{k}=0$.
By Assumption~\ref{asm:ambiguity}, there exists $\widehat{\vt{q}} \in \mathcal{U}$ such that $\widehat{\vt{q}} > \vt{0}$ and $\mathcal{L}(\vt{x}, \widehat{\vt{q}}) \geq \widehat{q}_{k} (-\log x_{k}) = \widehat{q}_{k} (-\log 0) = + \infty$.
Then, noting $+\infty \geq g(\vt{x}) = \sup_{\vt{p} \in \mathcal{U}} \mathcal{L}(\vt{x}, \vt{p}) \geq \mathcal{L}(\vt{x}, \widehat{\vt{q}})$ leads to~\eqref{cond:x}.
As a result, we get $\sup_{\vt{p} \in \mathcal{U}} \mathcal{L}(\vt{x}, \vt{p}) = \max_{\vt{p} \in \mathcal{U}} \mathcal{L}(\vt{x}, \vt{p})$ for $\vt{x} \in \Delta^n$.
Then, combining~\eqref{eq:infsup_H} and~\eqref{eq:infsup_minsup} yields $\min_{\vt{x} \in \Delta^n} \max_{\vt{p} \in \mathcal{U}} \mathcal{L}(\vt{x},\vt{p}) = \min_{\vt{x} \in \Delta^n} g(\vt{x}) = H(\vt{p}^{\ast}) < + \infty$.
Let $\vt{x}^{\ast}$ be any optimizer of $\min_{\vt{x} \in \Delta^n} \max_{\vt{p} \in \mathcal{U}} \mathcal{L}(\vt{x},\vt{p})$.
The definition of $g$ ensures $H(\vt{p}^{\ast}) = \min_{\vt{x} \in \Delta^n} g(\vt{x}) = g(\vt{x}^{\ast}) \geq \mathcal{L}(\vt{x}^{\ast}, \vt{p}^{\ast})$.
This fact and~\eqref{ineq:KL} imply $D_{\mathrm{KL}}(\vt{p}^{\ast}\,\|\,\vt{x}^{\ast}) = 0$, i.e., $\vt{x}^{\ast} = \vt{p}^{\ast}$.
It then follows from the uniqueness of $\vt{p}^{\ast}$ that the optimizer $\vt{x}^{\ast}$ is also unique.
Moreover, we obtain $\vt{x}^{\ast} \in \operatorname{ri}(\Delta^{n})$ by $g(\vt{x}^{\ast}) = H(\vt{p}^{\ast}) < +\infty$ because~\eqref{cond:x} holds.
Therefore, $x_j^{\ast} > 0$ for all $j \in N$.
\qed
\end{proof}

\begin{remark}
\label{rem:saddle}
Theorem~\ref{thm:unified} also shows that $\vt{p}^{*}$ attains the inner maximum at
$\vt{x}^{*}$.
Indeed, since $\vt{x}^{*} = \vt{p}^{*}$ and $H(\vt{p}) = \mathcal{L}(\vt{p},\vt{p})$, we have
\begin{align*}
  \mathcal{L}(\vt{x}^{*},\vt{p}^{*}) = H(\vt{p}^{*})
  = \max_{\vt{p} \in \mathcal{U}}\mathcal{L}(\vt{x}^{*},\vt{p}),
\end{align*}
where the second equality holds because $H(\vt{p}^{*})$ is the optimal value
of~\eqref{eq:DRO} and $\vt{x}^{*}$ is its optimal solution.
\end{remark}

Theorem~\ref{thm:unified} recasts DRO probability smoothing as an application of Jaynes's maximum entropy principle~\cite{J_PR1957}: among all distributions in the ambiguity set, the DRO estimator is the one of maximum entropy. The contrast between symmetric $q$-DRO and structure-respecting OT-DRO affects only the shape of the constraint set $\mathcal{U}$, not the underlying principle from which the estimator is derived.

This result is a specialization of the general minimax-entropy duality developed by Gr\"unwald and Dawid~\cite{GD_AS2004}, who established the equivalence between worst-case loss minimization and generalized entropy maximization for arbitrary loss functions. In the case of the cross-entropy loss on the probability simplex, the generalized entropy reduces to the Shannon entropy, and Theorem~\ref{thm:unified} follows as a direct consequence. While the underlying duality is thus classical, what we wish to highlight is the explanatory power of this specialization for the probability smoothing context, which we now demonstrate.

\subsection{Unified Explanations of Structural Properties}
The maximum entropy reformulation yields immediate consequences for several structural properties.

\paragraph{Existence, positivity and uniqueness.} By Theorem~\ref{thm:unified}, for any ambiguity set satisfying Assumption~\ref{asm:ambiguity}, the optimal estimator exists, is unique, and has strictly positive components.
This recovers in one argument the positivity of the $q$-DRO estimator and that of the OT-DRO estimator; cf.\ Proposition~\ref{prop:otdro-pos}, which gives the stronger statement that every feasible solution of~\eqref{eq:otdro-reform} has positive components.
These properties are not specific to either formulation but are universal consequences of the maximum entropy principle.

\paragraph{Threshold conditions.} The uniform distribution $\vt{1}_n/n$ is the global maximizer of $H$ over $\Delta^n$. Therefore, $\vt{x}^* = \vt{1}_n/n$ if and only if $\vt{1}_n/n \in \mathcal{U}(\hat{\vt{p}}, \varepsilon)$. For $\mathcal{U} = \mathcal{U}_q$, this condition is $\|\vt{1}_n/n - \hat{\vt{p}}\|_q \leq \varepsilon$, recovering Theorem~\ref{thm:qdro-threshold}. For $\mathcal{U} = \mathcal{U}_{\mathrm{OT}}$, it is $d_{\mathrm{OT}}(\vt{1}_n/n, \hat{\vt{p}}) \leq \varepsilon$, recovering Theorem~\ref{thm:otdro-threshold}. The parallel threshold structure observed in Remark~\ref{rem:parallel} is thus a direct consequence of the fact that the uniform distribution is the global entropy maximizer over the simplex.

\paragraph{Vanishing of dual multipliers in $q$-DRO.} In the convex reformulation of $q$-DRO~\cite{IKNU_ORL2026}, the nonnegativity constraints $\hat{p}_j + e_j \geq 0$ on the perturbation vector are associated with Lagrange multipliers $\lambda_j \geq 0$. It was shown in~\cite{IKNU_ORL2026} through a KKT analysis that $\lambda_j = 0$ at any optimal solution. The maximum entropy perspective yields an immediate proof: by the complementarity condition, $\lambda_j(\hat{p}_j + e_j^*) = 0$, where
$\vt{e}^{*} = \vt{p}^{*} - \hat{\vt{p}}$ is an optimal solution of the inner problem at
$\vt{x}^{*}$ by Remark~\ref{rem:saddle}; and $\hat{p}_j + e_j^* = p_j^* = x_j^* > 0$ follows from the positivity of the maximum entropy solution established above. The case-by-case argument in~\cite{IKNU_ORL2026} reduces to a special case of this general principle.

\section{Numerical Experiments}
\label{sec:experiments}
We conduct numerical experiments to validate the theoretical findings and to compare the practical performance of Laplace smoothing, $q$-DRO (with $q=2$, hereafter 2-DRO), and OT-DRO.
The central question is whether the geometry of the ambiguity set has a meaningful impact on smoothing quality: does OT-DRO's ability to incorporate inter-category structure translate into improved estimation when such structure is present, and does it degrade performance when no structure exists?
All experiments use MOSEK~\cite{Mosek_11} as the conic optimization solver, and are conducted on a machine with an Intel Core i7-14700F processor (2.10~GHz), 16~GB of memory, running 64-bit Windows~11.

\subsection{Threshold Visualization}
We first visualize the threshold phenomenon for both formulations using $n = 4$ categories and the empirical distribution $\hat{\vt{p}} = (0.10, 0.20, 0.30, 0.40)^\top$.
The degeneracy thresholds predicted by Theorems~\ref{thm:qdro-threshold} and~\ref{thm:otdro-threshold} are $\|\vt{1}_4/4 - \hat{\vt{p}}\|_2 = \sqrt{0.05}\approx 0.224$ for 2-DRO and $d_{\mathrm{OT}}(\vt{1}_4/4, \hat{\vt{p}}) = 1/6\approx 0.167$ for OT-DRO, where the cost matrix is $c_{ij} = |i-j|/(n-1)$.

Figure~\ref{fig:threshold} plots each component of the optimal estimator as the robustness radius $\varepsilon$ increases from $0$ up to and slightly beyond each threshold.
At $\varepsilon = 0$ the ambiguity set is the singleton $\{\hat{\vt{p}}\}$, and the DRO problem reduces to $\min_{\vt{x} \in \Delta^{n}}\mathcal{L}(\vt{x},\hat{\vt{p}})$, whose unique solution is $\hat{\vt{p}}$ itself by the nonnegativity of the Kullback--Leibler divergence, so that no smoothing takes place.
The empirical distribution used here is strictly positive, so Assumption~\ref{asm:ambiguity} still holds at $\varepsilon = 0$ and this left endpoint is well-defined.
We nevertheless assume $\varepsilon>0$ throughout, because for an empirical distribution with zero components the singleton $\{\hat{\vt{p}}\}$ does not meet the relative interior of $\Delta^{n}$ and the assumption fails.
As $\varepsilon$ increases, the components move toward the uniform value $0.25$, reaching it exactly at the predicted threshold (marked by the vertical dotted line) and remaining there for all larger $\varepsilon$.
This confirms Theorems~\ref{thm:qdro-threshold} and~\ref{thm:otdro-threshold}.
The two formulations differ in the trajectory taken toward the uniform distribution: the 2-DRO estimator moves along nearly straight lines, reflecting the symmetric $q^*$-norm penalty, whereas the OT-DRO estimator follows curved paths shaped by the transport cost between adjacent categories.
In particular, the OT-DRO component $x_2$ initially decreases slightly before rising toward $0.25$, a non-monotone behavior that does not occur for 2-DRO; this reflects that mass is redistributed locally between neighboring categories rather than symmetrically across all categories.

\begin{figure}[hbtp]
\begin{center}
\subfloat[$2$-DRO \label{fig:fig_q}]{
\includegraphics[scale=0.35]{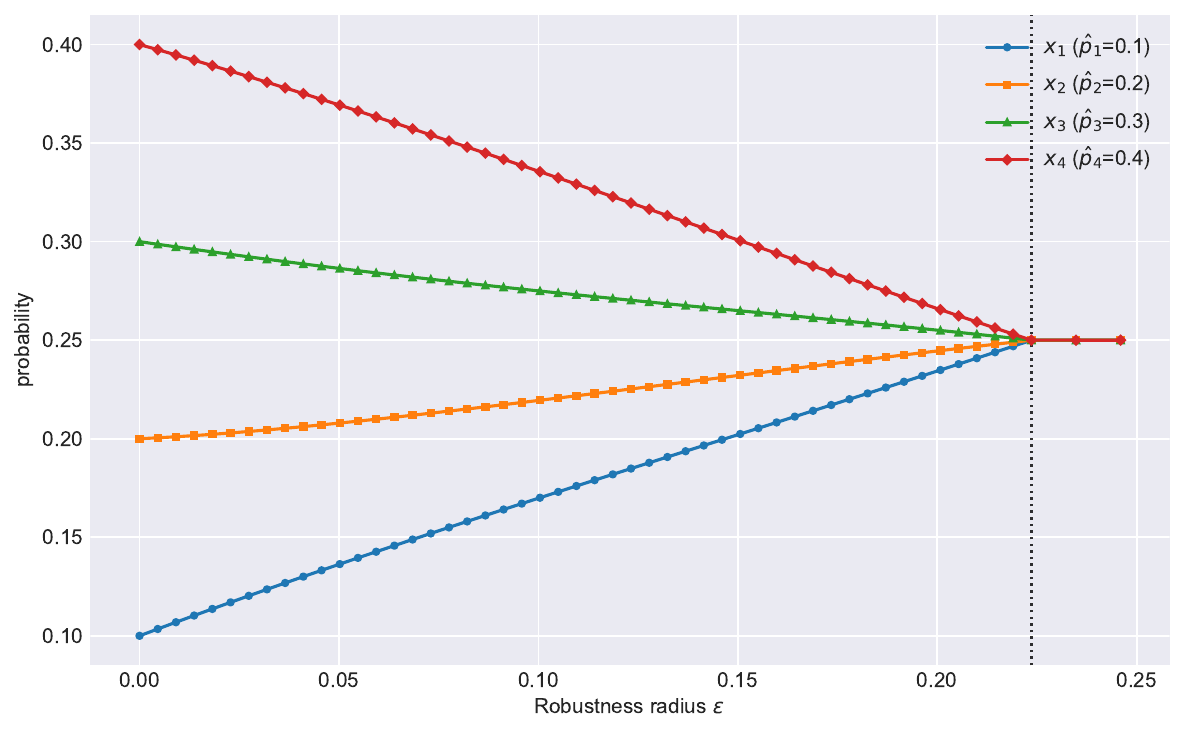}
}
\subfloat[OT-DRO \label{fig:fig_ot}]{
\includegraphics[scale=0.35]{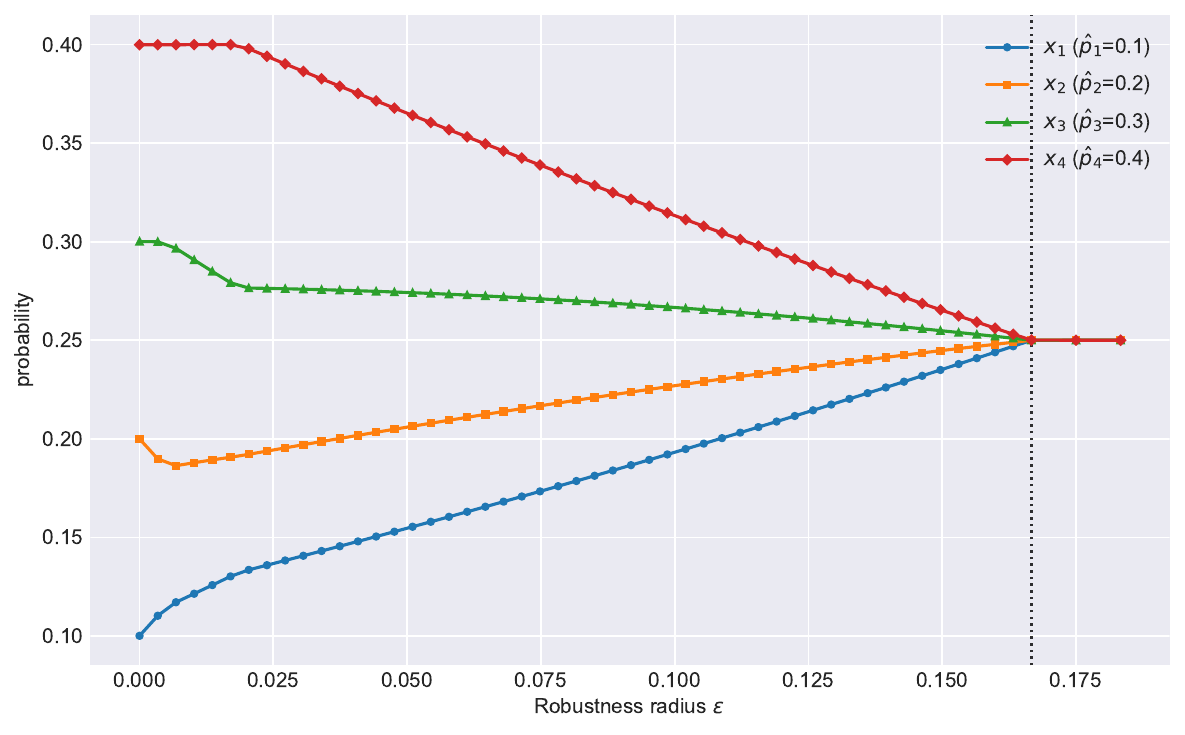}
}
\end{center}
\caption{Sensitivity of the 2-DRO and OT-DRO estimators $\vt{x}$ to the robustness radius.}\label{fig:threshold}
\end{figure}

\subsection{Structured vs.\ Unstructured Synthetic Data}
We next examine whether the choice of ambiguity set geometry has a meaningful impact on smoothing quality.
We consider discrete distributions over $n = 100$ categories and evaluate the three smoothing methods under four settings, grouped into structured and unstructured cases.
For OT-DRO, we use the ordinal cost matrix $c_{ij} = |i-j|/(n-1)$ throughout.

\paragraph{Structured settings.} 
The true distribution is a discretized mixture of three Gaussians on $[1, n]$, whose support has a natural ordinal structure that the cost matrix is designed to capture.
We consider two variants:
\begin{itemize}
    \item \textbf{Symmetric Structured:} The three Gaussian components are equally weighted with $\vt{w} = (1/3, 1/3, 1/3)$, means $\vt{\theta} = (0.2n, 0.5n, 0.8n)$, and standard deviations $\vt{\sigma} = (n/15, n/15, n/15)$.
    This represents a basic structured case where peaks and valleys are evenly distributed.
    \item \textbf{Asymmetric Structured:} The three components have weights $\vt{w} = (0.5, 0.3, 0.2)$, means $\vt{\theta} = (0.2n, 0.6n, 0.85n)$, and standard deviations $\vt{\sigma} = (n/10, n/20, n/20)$.
    This models a more complex practical situation where a broad major peak and sharp minor peaks coexist.
\end{itemize}

\paragraph{Unstructured settings.} 
The true distribution is drawn from a symmetric Dirichlet distribution $\mathrm{Dir}(\kappa \cdot \vt{1})$, whose support has no relation to the ordinal arrangement of the categories.
The same cost matrix $c_{ij} = |i-j|/(n-1)$ is applied to OT-DRO, testing whether it is penalized for assuming a structure that does not exist.
We test two concentration parameters:
\begin{itemize}
    \item \textbf{Dense Unstructured ($\kappa = 1.0$):} A broadly spread random distribution in which most categories have non-negligible probability.
    \item \textbf{Sparse Unstructured ($\kappa = 0.5$):} A distribution whose mass concentrates on a few categories, mimicking a severe zero-frequency regime. We emphasize that, although the empirical distribution is sparse here, the support has no ordinal structure; the sparsity concerns the number of observed categories, not the geometric arrangement that OT-DRO exploits.
\end{itemize}

For the unstructured settings, the true distribution is sampled once and kept fixed throughout the Monte Carlo repetitions.

For each repetition, we draw $m \in \{50, 100\}$ samples from the true distribution.
Since the empirical distribution $\hat{\vt{p}}$, and hence its degeneracy threshold, varies across repetitions, we parametrize the robustness radius by its ratio $\rho \in (0, 1)$ to the threshold, setting $\varepsilon = \rho\,\|\vt{1}_n/n - \hat{\vt{p}}\|_2$ for 2-DRO and $\varepsilon = \rho\,d_{\mathrm{OT}}(\vt{1}_n/n, \hat{\vt{p}})$ for OT-DRO.
Since $\rho < 1$, the estimator never degenerates to the uniform distribution, and the same $\rho$ represents the same relative strength of robustness across repetitions.
We select $\rho$ from the grid $\{0.1, 0.2, \ldots, 0.9\}$ and the Laplace smoothing parameter $a$ from ten logarithmically spaced values in $[10^{-4}, 1]$.
Each hyperparameter is chosen to minimize the KL divergence from the empirical distribution of a validation sample to the estimator.
The validation sample consists of $m_\text{val}=10^{4}$ observations drawn from the true distribution.
It is drawn once for each setting, independently of the training samples, and is shared by all the Monte Carlo repetitions.
All three methods are tuned in the same way.
Performance is measured by the KL divergence from the true distribution, averaged over 50 Monte Carlo repetitions.

\begin{table}[htbp]
\centering
\caption{Mean (standard deviation) of the KL divergence and computation time (in seconds) over 50 Monte Carlo repetitions. Bold indicates the best performance in KL divergence.}
\label{tab:main_results}
\resizebox{\textwidth}{!}{%
\begin{tabular}{llcccccc}
\toprule
 & & \multicolumn{2}{c}{Laplace} & \multicolumn{2}{c}{$2$-DRO} & \multicolumn{2}{c}{OT-DRO} \\
\cmidrule(lr){3-4} \cmidrule(lr){5-6} \cmidrule(lr){7-8}
Setting & $m$ & KL & Time (s) & KL & Time (s) & KL & Time (s) \\
\midrule
\multirow{2}{*}{Symmetric Structured} 
& 50 & 0.2034 (0.0121) & $<0.01$ (0.00) & 0.2056 (0.0124) & 0.10 (0.03) & \textbf{0.0945} (0.0315) & 1.54 (0.49) \\
& 100 & 0.1734 (0.0152) & $<0.01$ (0.00) & 0.1776 (0.0128) & 0.09 (0.01) & \textbf{0.0580} (0.0223) & 1.82 (0.26) \\
\midrule
\multirow{2}{*}{Asymmetric Structured} 
& 50 & 0.2168 (0.0129) & $<0.01$ (0.00) & 0.2189 (0.0131) & 0.10 (0.02) & \textbf{0.1009} (0.0364) & 1.41 (0.24) \\
& 100& 0.1844 (0.0162) & $<0.01$ (0.00) & 0.1899 (0.0143) & 0.09 (0.00) & \textbf{0.0658} (0.0198) & 1.91 (0.22) \\
\midrule
\multirow{2}{*}{Unstructured ($\kappa=1.0$)} 
 & 50 & \textbf{0.3059} (0.0345) & $<0.01$ (0.00) & 0.3103 (0.0310) & 0.11 (0.01) & 0.3965 (0.0341) & 1.42 (0.16) \\
& 100 & \textbf{0.2398} (0.0291) & $<0.01$ (0.00) & \textbf{0.2398} (0.0284) & 0.10 (0.01) & 0.3090 (0.0457) & 1.77 (0.15) \\
\midrule
\multirow{2}{*}{Unstructured ($\kappa=0.5$)} 
 &50 & \textbf{0.4163} (0.0604) & $<0.01$ (0.00) & 0.4189 (0.0607) & 0.10 (0.01) & 0.6399 (0.0711) & 1.14 (0.11) \\
& 100& \textbf{0.2845} (0.0455) & $<0.01$ (0.00) & 0.2870 (0.0483) & 0.10 (0.00) & 0.4516 (0.0739) & 1.51 (0.15) \\
\bottomrule
\end{tabular}%
}
\end{table}

Table~\ref{tab:main_results} summarizes the results.
The effectiveness of each estimator depends heavily on the alignment between the underlying data structure and the geometry of the ambiguity set.

\begin{figure}[htbp]
\centering
\includegraphics[width=\textwidth]{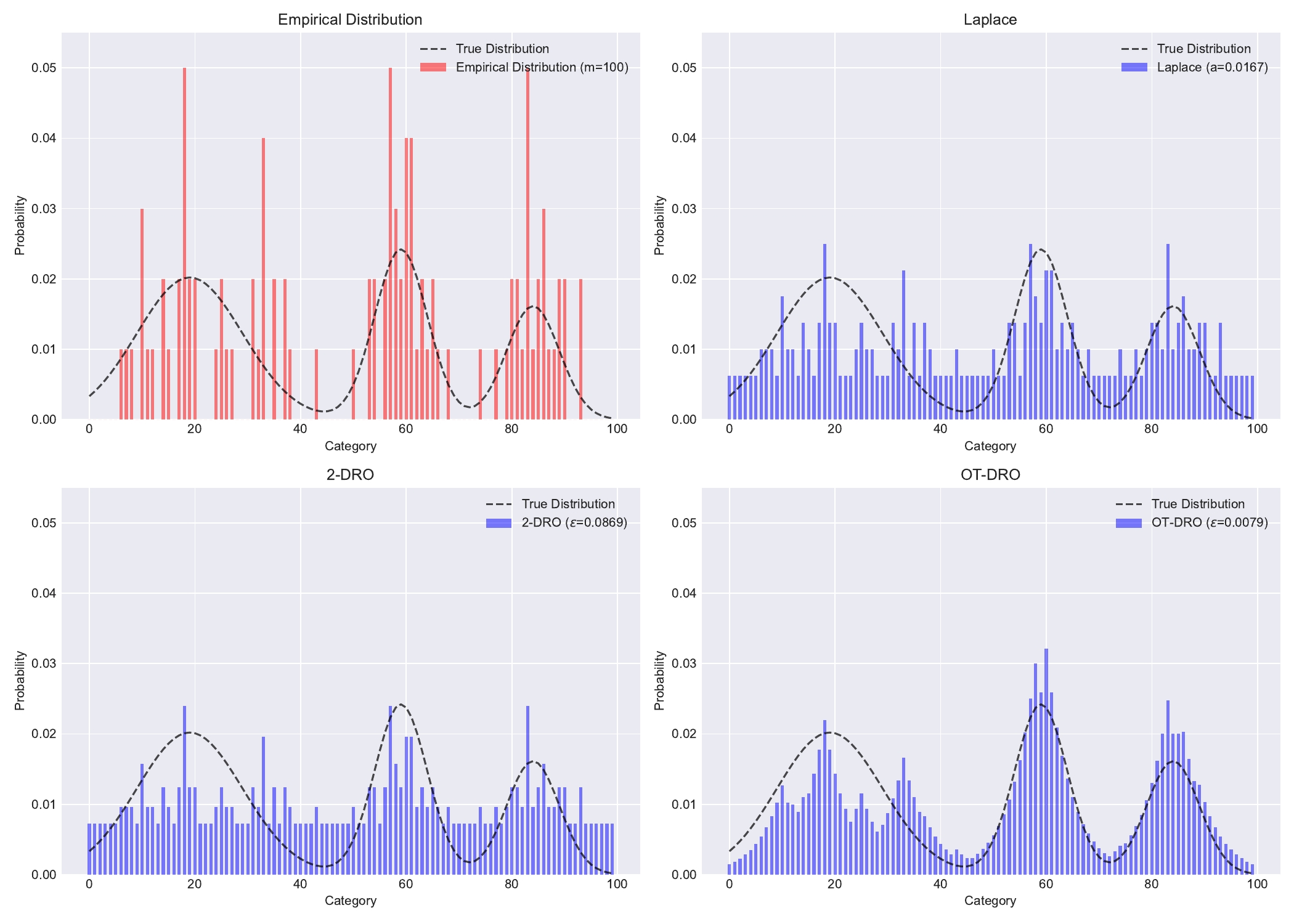}
\caption{Estimated distributions in the Asymmetric Structured setting ($n=100$, $m=100$).}
\label{fig:asym_struct}
\end{figure}

\begin{figure}[htbp]
\centering
\includegraphics[width=\textwidth]{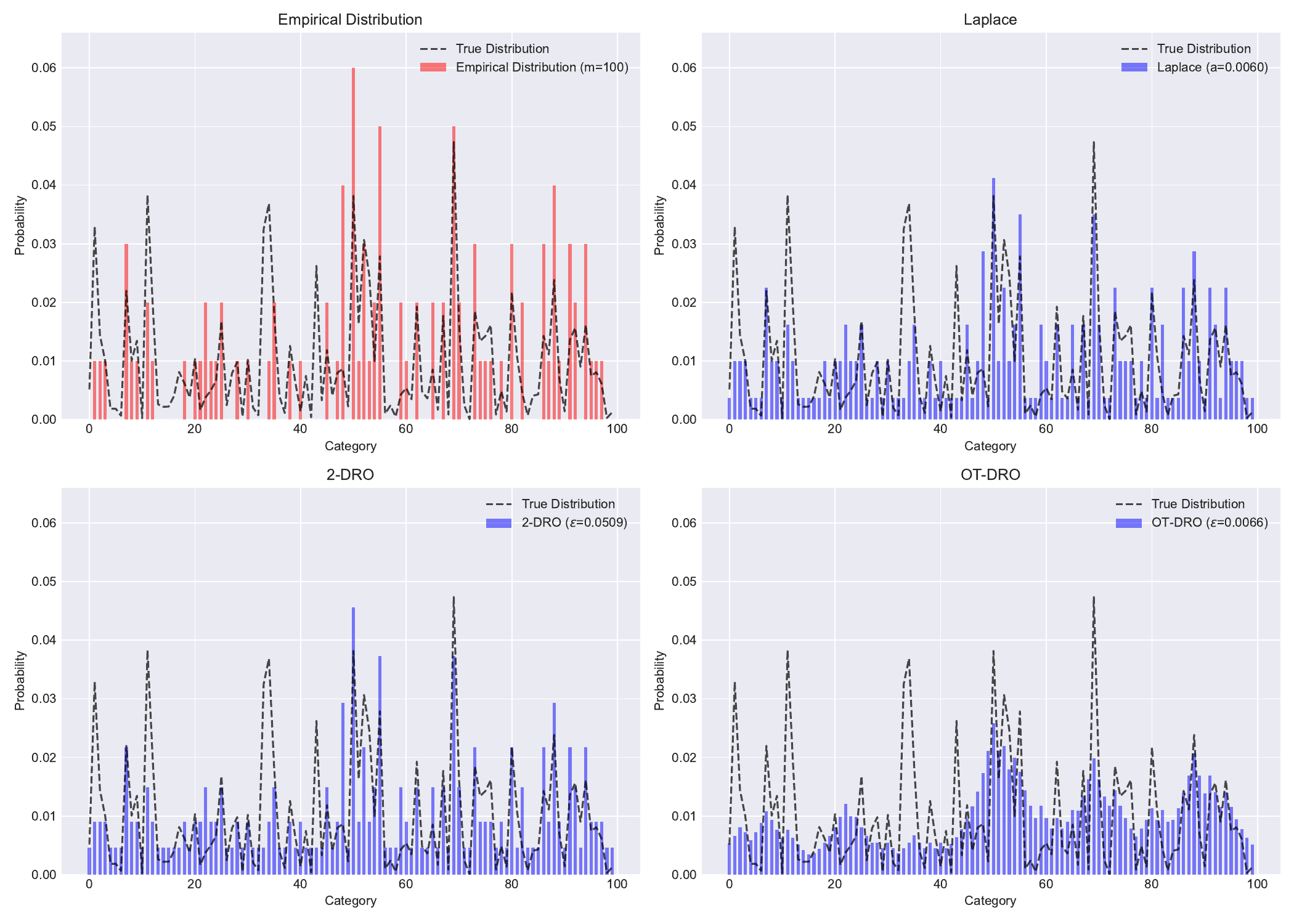}
\caption{Estimated distributions in the Dense Unstructured setting ($n=100$, $m=100$, $\kappa=1.0$).}
\label{fig:dense_unstruct}
\end{figure}

In the structured settings, OT-DRO substantially outperforms Laplace smoothing and 2-DRO, achieving less than half of their KL divergence.
Figure~\ref{fig:asym_struct} explains this advantage visually.
Given the noisy empirical distribution, Laplace smoothing and 2-DRO redistribute probability mass almost uniformly across all categories, flattening the estimate and blurring the three peaks of the true distribution.
In contrast, OT-DRO redistributes mass locally, guided by the cost matrix: it smooths within each peak while preserving the separation between peaks, thereby recovering a cleaner reconstruction of the underlying multimodal structure.
This behavior persists in the asymmetric case, where the broad major peak and the sharp minor peaks are captured without being averaged out.

In the unstructured settings, however, OT-DRO incurs a distinct penalty.
As Figure~\ref{fig:dense_unstruct} shows, the true Dirichlet distribution has no ordinal structure, so neighboring categories are unrelated.
By coupling adjacent categories through the cost matrix, OT-DRO oversmooths the estimate, dampening the sharp category-to-category fluctuations that the symmetric methods retain more faithfully.
Consequently, Laplace smoothing and 2-DRO produce lower KL divergence and remain the preferable choice when no reliable structural information is available.

This contrast is precisely the behavior predicted by the theory: the ambiguity set geometry should be matched to the structure of the support, and OT-DRO and $q$-DRO are complementary tools for the two regimes.

Finally, the computation-time columns of Table~\ref{tab:main_results} show that OT-DRO's accuracy comes at a computational cost: each OT-DRO instance takes approximately $1.5$ seconds, compared with approximately $0.1$ seconds for 2-DRO and negligible time for Laplace smoothing, reflecting the $O(n^2)$ constraints in its reformulation~\eqref{eq:otdro-reform}.

\section{Conclusion}
\label{sec:conclusion}
In this paper, we studied distributionally robust optimization for the probability smoothing problem from a unifying viewpoint.
Motivated by the observation that Symmetry forces every unobserved category to receive identical probability, we introduced an optimal transport distributionally robust formulation (OT-DRO) that incorporates inter-category structure through a cost matrix, thereby enabling structure-respecting smoothing.
We derived a tractable reformulation of OT-DRO as an empirical cross-entropy minimization with a Lipschitz-type penalty on the log-probability vector, established the positivity of the resulting estimator, and identified the explicit degeneracy threshold $d_{\mathrm{OT}}(\vt{1}_n/n, \hat{\vt{p}})$, which parallels the threshold $\|\vt{1}_n/n - \hat{\vt{p}}\|_q$ of $q$-DRO.

The central message of the paper is that these parallel structures are not coincidental.
Through Sion's minimax theorem, we showed that distributionally robust probability smoothing over any compact convex ambiguity set that contains a strictly positive distribution is equivalent to entropy maximization over that set.
This maximum entropy principle explains, within a single framework, the positivity of the estimator, its uniqueness, and the threshold conditions.
It also clarifies that the contrast between the symmetric $q$-DRO and the structure-respecting OT-DRO lies entirely in the shape of the ambiguity set, not in the estimation principle itself.
Our numerical experiments confirmed this view: OT-DRO outperforms the symmetric methods when the underlying distribution respects the category structure, while the symmetric methods remain preferable when no such structure exists.

\newpage
\appendix

\section{Omitted Proofs}\label{app:threshold}
\subsection{Proof of Threshold Condition of $q$-DRO}
\label{app:threshold_qdro}
By Theorem~\ref{thm:qdro-equiv}, $\vt{x}^*$ is optimal for the $q$-DRO problem~\eqref{eq:qdro} if and only if there exists $\beta^*$ such that $(\vt{x}^*, \beta^*)$ is optimal for~\eqref{eq:qdro-simplified}. We therefore establish the threshold condition for~\eqref{eq:qdro-simplified}.
Here, $\partial$ denotes the convex subdifferential, and the subgradient of the composite penalty is obtained by the chain rule for the composition of the convex norm with the continuously differentiable map $-\log(\vt{x}) - \beta\vt{1}_n$.

The proof relies on the following standard fact about the subdifferential of a norm.
\begin{lemma}[{\cite[Chap.~VI, \S3]{HU_2011}}]
\label{lem:subgrad}
Let $\|\cdot\|_{q^*}$ be the $q^*$-norm on $\mathbb{R}^n$, where $q^*$ is the H\"older conjugate of $q$ satisfying $1/q + 1/q^* = 1$. Then, for any $\vt{u} \in \mathbb{R}^n$, a vector $\vt{s}$ belongs to $\partial\|\cdot\|_{q^*}(\vt{u})$ if and only if $\|\vt{s}\|_q \leq 1$ and $\vt{s}^\top\vt{u} = \|\vt{u}\|_{q^*}$. In particular, at $\vt{u} = \vt{0}$, the second condition holds trivially, so $\vt{s} \in \partial\|\cdot\|_{q^*}(\vt{0})$ if and only if $\|\vt{s}\|_q \leq 1$.
\end{lemma}

We now derive the optimality conditions for $q$-DRO~\eqref{eq:qdro-simplified}.
The Lagrangian $L(\vt{x},\beta,\gamma)$ for $q$-DRO is given by
\begin{align*}
  L(\vt{x},\beta,\gamma)&=\sum_{j=1}^{n}\hat{p}_{j}(-\log{x_{j}})+\varepsilon\cdot \|-\log(\vt{x})-\beta \vt{1}_{n}\|_{q^{*}}+\gamma\left(\sum_{j=1}^{n}x_{j}-1 \right).
\end{align*}
Note that we do not need to introduce multipliers for the nonnegativity constraints ($x_{j} \ge 0$): by Theorem~\ref{thm:qdro-equiv} the $\vt{x}$-component of an optimal solution of~\eqref{eq:qdro-simplified} is optimal for~\eqref{eq:qdro}, and Theorem~\ref{thm:unified} then gives $x_{j}>0$ for all $j \in N$.

The generalized KKT conditions state that at an optimal solution $(\vt{x},\beta)$, there exists a multiplier $\gamma$ and a subgradient $\vt{s} \in \partial(\| -\log(\vt{x})-\beta \vt{1}_{n} \|_{q^{*}})$ satisfying the following conditions:
\begin{align}
  -\frac{\hat{p}_{j}}{x_{j}}+\gamma-\frac{\varepsilon s_{j}}{x_{j}}&=0 \quad (j \in N), \label{eq:kkt1_q}\\
  \varepsilon \sum_{j=1}^{n}s_{j}&=0. \label{eq:kkt2_q}
\end{align}
Multiplying the first condition~\eqref{eq:kkt1_q} by $x_{j}$ and summing over all $j \in N$ yields $-\sum_{j=1}^{n}\hat{p}_{j} + \gamma\sum_{j=1}^{n}x_{j} - \varepsilon\sum_{j=1}^{n}s_{j} = 0$.
Since $\sum_{j=1}^{n}\hat{p}_{j} = 1$, $\sum_{j=1}^{n}x_{j} = 1$ and~\eqref{eq:kkt2_q} hold, we obtain $\gamma = 1$.
Thus, the KKT conditions require $\varepsilon s_{j} = x_{j} - \hat{p}_{j}$ for all $j \in N$. 

\begin{proof}[of Theorem~\ref{thm:qdro-threshold}]
  ($\Rightarrow$) Suppose that the uniform distribution $\vt{x}^* = \vt{1}_{n}/n$ is an optimal solution of $q$-DRO.
  By setting $\beta^* = \log n$, which makes the penalty term vanish, we have $-\log(\vt{x}^*) - \beta^*\vt{1}_{n} = \vt{0}$.
  By Lemma~\ref{lem:subgrad} applied at the origin $\vt{0}$, the subgradient $\vt{s} \in \partial\|\cdot\|_{q^*}(\vt{0})$ satisfies $\|\vt{s}\|_q \leq 1$.
  Substituting $\vt{x}^* = \vt{1}_{n}/n$ into the relation $\varepsilon\vt{s} = \vt{x}^* - \hat{\vt{p}}$ derived from the KKT conditions, we obtain $\varepsilon\vt{s} = \vt{1}_{n}/n - \hat{\vt{p}}$.
  Taking the $q$-norm of both sides, we obtain $\|\vt{1}_{n}/n - \hat{\vt{p}}\|_q = \varepsilon\|\vt{s}\|_q \le \varepsilon$.

  ($\Leftarrow$) Conversely, suppose $\|\vt{1}_{n}/n - \hat{\vt{p}}\|_q \le \varepsilon$.
  By H\"older's inequality, for any vector $\vt{u} \in \mathbb{R}^n$, we have
  \begin{align*}
    (\vt{1}_{n}/n-\hat{\vt{p}})^{\top}\vt{u} \le \|\vt{1}_{n}/n-\hat{\vt{p}}\|_{q}\|\vt{u}\|_{q^{*}} \le \varepsilon\|\vt{u}\|_{q^{*}}.
  \end{align*}
  Setting $\vt{u} = -\log(\vt{x}) - \beta\vt{1}_{n}$, the objective function value for any feasible $(\vt{x}, \beta)$ is bounded below as follows:
  \begin{align*}
    \sum_{j=1}^{n}\hat{p}_{j}(-\log{x_{j}})+\varepsilon\cdot \|-\log(\vt{x})-\beta\vt{1}_{n}\|_{q^{*}}
    &\ge \sum_{j=1}^{n}\hat{p}_{j}(-\log{x_{j}})+(\vt{1}_{n}/n-\hat{\vt{p}})^{\top}(-\log(\vt{x})-\beta\vt{1}_{n})\\
    &=\frac{1}{n}\sum_{j=1}^{n}(-\log{x_{j}})\\
    &\ge -\log \left(\frac{1}{n}\sum_{j=1}^{n}x_{j} \right)
    =\log{n},
  \end{align*}
  where the last inequality follows from Jensen's inequality. Thus, the optimal value of~\eqref{eq:qdro-simplified} is at least $\log n$. 
  
  On the other hand, the feasible solution $(\vt{x}, \beta) = (\vt{1}_{n}/n, \log n)$ attains the objective value $\log n$. Therefore, the uniform distribution $\vt{x} = \vt{1}_{n}/n$ is an optimal solution for~\eqref{eq:qdro-simplified}.
  By Theorem~\ref{thm:qdro-equiv}, the uniform distribution $\vt{1}_n/n$ is therefore an optimal solution of the $q$-DRO model~\eqref{eq:qdro}.
  \qed
\end{proof}

\subsection{Proof of Lemma~\ref{lem:otdro-set}}
\label{app:ambiguity}
We show that $\mathcal{U}_{\mathrm{OT}}(\hat{\vt{p}},\varepsilon)$ satisfies the conditions in Assumption~\ref{asm:ambiguity}.
Since $\mathcal{U}_{\mathrm{OT}}(\hat{\vt{p}},\varepsilon)$ is a subset of $\Delta^{n}$, it is bounded.
Thus, it is compact if it is closed.
The main part of the proof is therefore to show convexity and closedness.

\begin{proof}[Proof of Lemma~\ref{lem:otdro-set}]
In this proof, we write $\langle C,\vt{y}\rangle = \sum_{i=1}^{n}\sum_{j=1}^{n}c_{ij}y_{ij}$ for the cost of a transport plan $\vt{y}$, and $d(\vt{p}) = d_{\mathrm{OT}}(\vt{p},\hat{\vt{p}})$.
Recall that $\vt{y}$ is a transport plan from $\hat{\vt{p}}$ to $\vt{p}$ if
\begin{align}
\label{eq:plan}
  y_{ij} \ge 0 \ (i,j \in N), \qquad
  \sum_{j=1}^{n}y_{ij} = \hat{p}_{i} \ (i \in N), \qquad
  \sum_{i=1}^{n}y_{ij} = p_{j} \ (j \in N),
\end{align}
and that $d(\vt{p})$ is the minimum of $\langle C,\vt{y}\rangle$ over all such plans.
For each $\vt{p} \in \Delta^{n}$, this minimum is attained.
Indeed, it is the optimal value of an LP whose feasible region is nonempty and bounded.
The feasible region is nonempty because it contains the plan $y_{ij} = \hat{p}_{i}p_{j}$, and it is bounded because~\eqref{eq:plan} implies $0 \le y_{ij} \le \hat{p}_{i} \le 1$.

\emph{Nonemptiness.}
Let $y_{ij} = \hat{p}_{i}\delta_{ij}$, where $\delta_{ij}$ is the Kronecker delta.
Then $\vt{y}$ is a transport plan from $\hat{\vt{p}}$ to $\hat{\vt{p}}$, and $\langle C,\vt{y}\rangle = \sum_{i=1}^{n}c_{ii}\hat{p}_{i} = 0 \le \varepsilon$.
Hence, $d(\hat{\vt{p}}) = 0$ and $\hat{\vt{p}} \in \mathcal{U}_{\mathrm{OT}}(\hat{\vt{p}},\varepsilon)$.

\emph{Convexity.}
Let $\vt{p}^{1},\vt{p}^{2} \in \Delta^{n}$ and $t \in [0,1]$.
Let $\vt{y}^{1}$ and $\vt{y}^{2}$ be optimal transport plans from $\hat{\vt{p}}$ to $\vt{p}^{1}$ and to $\vt{p}^{2}$, respectively.
Since the constraints in~\eqref{eq:plan} are linear, we have
\begin{align*}
  t y^{1}_{ij} + (1-t) y^{2}_{ij} \ge 0, \qquad
  \sum_{j=1}^{n}\big(ty^{1}_{ij}+(1-t)y^{2}_{ij}\big) = \hat{p}_{i}, \qquad
  \sum_{i=1}^{n}\big(ty^{1}_{ij}+(1-t)y^{2}_{ij}\big) = tp^{1}_{j}+(1-t)p^{2}_{j}.
\end{align*}
Thus, $t\vt{y}^{1}+(1-t)\vt{y}^{2}$ is a transport plan from $\hat{\vt{p}}$ to $t\vt{p}^{1}+(1-t)\vt{p}^{2}$.
Since the cost $\langle C,\cdot\rangle$ is linear, we obtain
\begin{align*}
  d\big(t\vt{p}^{1}+(1-t)\vt{p}^{2}\big)
  \le \big\langle C,\, t\vt{y}^{1}+(1-t)\vt{y}^{2}\big\rangle
  = t\langle C,\vt{y}^{1}\rangle + (1-t)\langle C,\vt{y}^{2}\rangle
  = t\,d(\vt{p}^{1}) + (1-t)\,d(\vt{p}^{2}).
\end{align*}
Therefore, $d$ is convex on $\Delta^{n}$, and its sublevel set $\mathcal{U}_{\mathrm{OT}}(\hat{\vt{p}},\varepsilon)$ is convex.

\emph{Closedness.}
Let $\{\vt{p}^{k}\} \subseteq \mathcal{U}_{\mathrm{OT}}(\hat{\vt{p}},\varepsilon)$ be a sequence converging to $\bar{\vt{p}}$.
For each $k$, let $\vt{y}^{k}$ be a transport plan from $\hat{\vt{p}}$ to $\vt{p}^{k}$ such that $\langle C,\vt{y}^{k}\rangle \le \varepsilon$.
Since $0 \le y^{k}_{ij} \le \hat{p}_{i} \le 1$, the sequence $\{\vt{y}^{k}\}$ is bounded.
Thus, it has a subsequence $\{\vt{y}^{k_{m}}\}$ that converges to some $\bar{\vt{y}}$.
We have
\begin{align*}
  y^{k_{m}}_{ij} \ge 0, \qquad
  \sum_{j=1}^{n}y^{k_{m}}_{ij} = \hat{p}_{i}, \qquad
  \sum_{i=1}^{n}y^{k_{m}}_{ij} = p^{k_{m}}_{j}, \qquad
  \langle C,\vt{y}^{k_{m}}\rangle \le \varepsilon .
\end{align*}
Letting $m \to \infty$, we obtain
\begin{align*}
  \bar{y}_{ij} \ge 0, \qquad
  \sum_{j=1}^{n}\bar{y}_{ij} = \hat{p}_{i}, \qquad
  \sum_{i=1}^{n}\bar{y}_{ij} = \bar{p}_{j}, \qquad
  \langle C,\bar{\vt{y}}\rangle \le \varepsilon .
\end{align*}
This means that $\bar{\vt{y}}$ is a transport plan from $\hat{\vt{p}}$ to $\bar{\vt{p}}$ whose cost is at most $\varepsilon$.
Hence, $d(\bar{\vt{p}}) \le \varepsilon$.
Since $\Delta^{n}$ is closed, we also have $\bar{\vt{p}} \in \Delta^{n}$.
Therefore, $\bar{\vt{p}} \in \mathcal{U}_{\mathrm{OT}}(\hat{\vt{p}},\varepsilon)$, and the set is closed.

\emph{Intersection with the relative interior.}
For $t \in (0,1]$, let $\vt{p}_{t} = (1-t)\hat{\vt{p}} + t\,\vt{1}_{n}/n$.
All the components of $\vt{p}_{t}$ are at least $t/n>0$.
We define
\begin{align*}
  y^{t}_{ij} = (1-t)\hat{p}_{i}\delta_{ij} + \frac{t\,\hat{p}_{i}}{n} \qquad (i,j \in N).
\end{align*}
Then $y^{t}_{ij} \ge 0$, and
\begin{align*}
  \sum_{j=1}^{n}y^{t}_{ij} = (1-t)\hat{p}_{i} + t\,\hat{p}_{i} = \hat{p}_{i},
  \qquad
  \sum_{i=1}^{n}y^{t}_{ij} = (1-t)\hat{p}_{j} + \frac{t}{n} = (\vt{p}_{t})_{j}.
\end{align*}
Thus, $\vt{y}^{t}$ is a transport plan from $\hat{\vt{p}}$ to $\vt{p}_{t}$.
Since $c_{ii}=0$ and $\sum_{i}\hat{p}_{i}=1$, its cost satisfies
\begin{align*}
  \langle C,\vt{y}^{t}\rangle
  = \frac{t}{n}\sum_{i=1}^{n}\hat{p}_{i}\sum_{j=1}^{n}c_{ij}
  \le t\max_{i,j \in N}\{ c_{ij} \}.
\end{align*}
The right-hand side converges to $0$ as $t \to +0$.
Hence, $d(\vt{p}_{t}) \le \varepsilon$ for all sufficiently small $t>0$.
For such $t$, the point $\vt{p}_{t}$ belongs to $\mathcal{U}_{\mathrm{OT}}(\hat{\vt{p}},\varepsilon)$, and all of its components are positive.\qed
\end{proof}

\subsection{Proof of Theorem~\ref{thm:otdro-equiv}}
\label{app:solvable}
We first recall the two problems in the theorem.
The OT-DRO model~\eqref{eq:otdro} and its reformulation~\eqref{eq:otdro-reform} can be written as the following problems:
\begin{align*}
  \text{\eqref{eq:otdro}:} \quad
  & \min_{\vt{x} \in \Delta^{n}} \;
    \underbrace{\max_{\vt{p} \in \mathcal{U}_{\mathrm{OT}}(\hat{\vt{p}},\varepsilon)}
    \mathcal{L}(\vt{x},\vt{p})}_{=:\, g(\vt{x})}, \\
  \text{\eqref{eq:otdro-reform}:} \quad
  & \min_{\vt{x} \in \Delta^{n}} \;
    \underbrace{\inf_{\alpha \ge 0,\, \vt{\mu} \in \bb{R}^{n}}
    \left\{\sum_{i=1}^{n}\hat{p}_{i}\mu_{i}+\varepsilon\alpha \;\middle|\;
    -\log x_{j}-c_{ij}\alpha \le \mu_{i} \ (i,j \in N)\right\}}_{=:\, h(\vt{x})}.
\end{align*}
For a fixed $\vt{x} \in \Delta^{n}$, the inner problem of~\eqref{eq:otdro} is problem~\eqref{eq:otdro-inner-p}, and it has the same optimal value as problem~\eqref{eq:otdro-inner} by Section~\ref{subsec:inner}.
The inner problem of~\eqref{eq:otdro-reform} is problem~\eqref{eq:otdro-inner-dual}, which is the dual problem of~\eqref{eq:otdro-inner} derived in Section~\ref{subsec:dual}.
Hence, by Lemma~\ref{lem:otdro-inner-dual}, we have
\begin{align}
\label{eq:otdro-gh}
  g(\vt{x})
  = \text{optimal value of~\eqref{eq:otdro-inner-p}}
  = \text{optimal value of~\eqref{eq:otdro-inner}}
  = \text{optimal value of~\eqref{eq:otdro-inner-dual}}
  = h(\vt{x})
  \qquad (\vt{x} \in \Delta^{n}).
\end{align}
If some component of $\vt{x}$ is zero, all the values in~\eqref{eq:otdro-gh} are $+\infty$.
The proof below is based on~\eqref{eq:otdro-gh} and on Theorem~\ref{thm:unified}, which gives an optimal solution of~\eqref{eq:otdro}.
Note that Theorem~\ref{thm:unified} holds for any ambiguity set that satisfies Assumption~\ref{asm:ambiguity}, and it does not depend on the reformulation~\eqref{eq:otdro-reform}.

\begin{proof}[of Theorem~\ref{thm:otdro-equiv}]
Let $F$ be the feasible region of~\eqref{eq:otdro-reform}, and let $f(\vt{x},\alpha,\vt{\mu}) = \sum_{i=1}^{n}\hat{p}_{i}\mu_{i}+\varepsilon\alpha$ be its objective function.
Note that $f$ does not depend on $\vt{x}$, and it is also the objective function
of~\eqref{eq:otdro-inner-dual}.
Since the infimum of $f(\vt{x},\alpha,\vt{\mu})$ over $F$ can be taken first with respect to $(\alpha,\vt{\mu})$
and then with respect to $\vt{x}$, we have
\begin{align}
\label{eq:otdro-reform-value}
  \inf_{(\vt{x},\alpha,\vt{\mu})\in F}f(\vt{x},\alpha,\vt{\mu}) = \inf_{\vt{x} \in \Delta^{n}}h(\vt{x}) .
\end{align}

\emph{Step 1: equality of the optimal values.}
Taking the infimum of~\eqref{eq:otdro-gh} over $\vt{x} \in \Delta^{n}$ and using~\eqref{eq:otdro-reform-value}, we obtain
\begin{align}
\label{eq:otdro-values}
  \inf_{\vt{x} \in \Delta^{n}}g(\vt{x}) = \inf_{\vt{x} \in \Delta^{n}}h(\vt{x}) = \inf_{(\vt{x},\alpha,\vt{\mu})\in F}f(\vt{x},\alpha,\vt{\mu}) .
\end{align}
Thus, the OT-DRO model~\eqref{eq:otdro} and problem~\eqref{eq:otdro-reform} have the same optimal value.

\emph{Step 2: solvability and the ``only if'' part.}
By Lemma~\ref{lem:otdro-set}, the set $\mathcal{U}_{\mathrm{OT}}(\hat{\vt{p}},\varepsilon)$ satisfies Assumption~\ref{asm:ambiguity}.
Hence, by Theorem~\ref{thm:unified}, the OT-DRO model~\eqref{eq:otdro} has a unique optimal solution $\vt{x}^{*}$, and $x^{*}_{j}>0$ for all $j \in N$.
By Lemma~\ref{lem:otdro-inner-dual}, problem~\eqref{eq:otdro-inner-dual} with $\vt{x} = \vt{x}^{*}$ has an optimal solution $(\alpha^{*},\vt{\mu}^{*})$.
Since $\vt{x}^{*} \in \Delta^{n}$ and $(\alpha^{*},\vt{\mu}^{*})$ satisfies the constraints of~\eqref{eq:otdro-inner-dual} with $\vt{x} = \vt{x}^{*}$, the point
$(\vt{x}^{*},\alpha^{*},\vt{\mu}^{*})$ belongs to $F$.
Then
\begin{align*}
  f(\vt{x}^{*},\alpha^{*},\vt{\mu}^{*}) = h(\vt{x}^{*}) \overset{\eqref{eq:otdro-gh}}{=} g(\vt{x}^{*})
  = \min_{\vt{x} \in \Delta^{n}}g(\vt{x}) \overset{\eqref{eq:otdro-values}}{=} \inf_{(\vt{x},\alpha,\vt{\mu})\in F}f(\vt{x},\alpha,\vt{\mu}).
\end{align*}
Hence, $(\vt{x}^{*},\alpha^{*},\vt{\mu}^{*})$ is an optimal solution of~\eqref{eq:otdro-reform}.
This shows that both problems have optimal solutions.
It also shows the ``only if'' part of the last statement, because $\vt{x}^{*}$ is the unique optimal solution of~\eqref{eq:otdro}.

\emph{Step 3: the ``if'' part.}
Suppose that $(\bar{\vt{x}},\bar{\alpha},\bar{\vt{\mu}})$ is an optimal solution of~\eqref{eq:otdro-reform}.
Since $(\bar{\alpha},\bar{\vt{\mu}})$ is a feasible solution of~\eqref{eq:otdro-inner-dual} with $\vt{x} = \bar{\vt{x}}$, we have $h(\bar{\vt{x}}) \le f(\bar{\vt{x}},\bar{\alpha},\bar{\vt{\mu}})$.
Hence,
\begin{align*}
  g(\bar{\vt{x}}) \overset{\eqref{eq:otdro-gh}}{=} h(\bar{\vt{x}}) \le f(\bar{\vt{x}},\bar{\alpha},\bar{\vt{\mu}}) = \inf_{(\vt{x},\alpha,\vt{\mu})\in F}f(\vt{x},\alpha,\vt{\mu})
  \overset{\eqref{eq:otdro-values}}{=} \inf_{\vt{x} \in \Delta^{n}}g(\vt{x}) \le g(\bar{\vt{x}}).
\end{align*}
Thus, all the inequalities hold with equality.
Therefore, $\bar{\vt{x}}$ is an optimal solution of~\eqref{eq:otdro}.\qed
\end{proof}

\subsection{Proof of Threshold Condition of OT-DRO}
\label{app:threshold_otdro}
By Theorem~\ref{thm:otdro-equiv}, $\vt{x}^*$ is optimal for OT-DRO if and only if there exist $(\alpha^*, \vt{\mu}^*)$ such that $(\vt{x}^*, \alpha^*, \vt{\mu}^*)$ is optimal for~\eqref{eq:otdro-reform}. We therefore establish the threshold condition for~\eqref{eq:otdro-reform}. Since~\eqref{eq:otdro-reform} is convex and satisfies Slater's condition, its KKT conditions are necessary and sufficient for optimality.

We introduce the Lagrangian for problem~\eqref{eq:otdro-reform}.
Let $\gamma\in \bb{R}$ be the Lagrange multiplier for the constraint~\eqref{eq:otdro-reform-con1}, $\pi_{ij}\ge0$ be the multiplier for the constraint~\eqref{eq:otdro-reform-con2}, and $\omega \ge 0$ be the multiplier for the constraint $\alpha \ge 0$,
then the Lagrangian $L(\vt{x},\alpha,\vt{\mu}, \gamma, \vt{\pi},\omega)$ for problem~\eqref{eq:otdro-reform} is given by
\begin{align*}
  L(\vt{x},\alpha,\vt{\mu}, \gamma, \vt{\pi},\omega)
  &=\sum_{i=1}^{n}\hat{p}_{i}\mu_{i}+\varepsilon \alpha +\gamma\left(\sum_{j=1}^{n}x_{j}-1\right)-\sum_{i=1}^{n}\sum_{j=1}^{n}\pi_{ij}(\mu_{i}+\log{x_{j}}+c_{ij}\alpha)-\alpha\omega.
\end{align*}
By Proposition~\ref{prop:otdro-pos}, any optimal solution satisfies $x_j>0$, so no multipliers for $x_{j}\ge 0$ are needed.

At an optimal solution of~\eqref{eq:otdro-reform}, there exist multipliers $(\gamma, \vt{\pi},\omega)$ satisfying the following KKT conditions:
\begin{align*}
\gamma - \sum_{i=1}^n\frac{\pi_{ij}}{x_j} &= 0 \quad (j \in N),\\
\varepsilon - \sum_{i=1}^n \sum_{j=1}^n c_{ij} \pi_{ij}-\omega &= 0,\\
\hat{p}_{i} - \sum_{j=1}^{n} \pi_{ij} &= 0 \quad (i \in N),\\
\pi_{ij} \left(\mu_i+\log x_j +c_{ij}\alpha \right)&= 0 \quad (i, j \in N),\\
\alpha \omega &= 0.
\end{align*}

Multiplying the first condition by $x_{j}$ and summing over $j \in N$ gives $\gamma\sum_{j=1}^{n}x_{j} = \sum_{j=1}^{n}\sum_{i=1}^{n}\pi_{ij} = \sum_{i=1}^{n}\hat{p}_{i} = 1$ by the third condition; since $\sum_{j=1}^{n}x_{j}=1$, it follows that $\gamma = 1$.
Thus, the optimal solution satisfies the following equations:
\begin{align}
x_j &= \sum_{i=1}^n \pi_{ij} \quad (j \in N),\label{eq:ot_kkt1_app}\\
\varepsilon - \omega &= \sum_{i=1}^n \sum_{j=1}^n c_{ij} \pi_{ij},\label{eq:ot_kkt2_app}\\
\hat{p}_{i} &= \sum_{j=1}^{n} \pi_{ij} \quad (i \in N).\label{eq:ot_kkt3_app}
\end{align}
Equations~\eqref{eq:ot_kkt1_app} and \eqref{eq:ot_kkt3_app} indicate that the Lagrange multipliers $\pi_{ij}$ form a transport plan $\vt{\pi} = (\pi_{ij}) \in \bb{R}^{n \times n}$ from the empirical distribution $\hat{\vt{p}}$ to the optimal solution $\vt{x}$.

\begin{proof}[of Theorem~\ref{thm:otdro-threshold}]
($\Rightarrow$) Suppose that the optimal solution of~\eqref{eq:otdro-reform} is the uniform distribution $\vt{x} = \vt{1}_n/n$. Since it satisfies the KKT conditions \eqref{eq:ot_kkt1_app}, \eqref{eq:ot_kkt2_app}, and \eqref{eq:ot_kkt3_app}, we have
\begin{align*}
  d_{\mathrm{OT}}(\vt{1}_n/n, \hat{\vt{p}}) \le \sum_{i=1}^{n}\sum_{j=1}^{n}c_{ij}\pi_{ij} = \varepsilon - \omega \le \varepsilon.
\end{align*}

($\Leftarrow$) Conversely, suppose $d_{\mathrm{OT}}(\vt{1}_n/n, \hat{\vt{p}}) \le \varepsilon$.
Then, there exists a transport plan from $\hat{\vt{p}}$ to $\vt{1}_n/n$ satisfying
\begin{align*}
  \sum_{i=1}^{n}\sum_{j=1}^{n}c_{ij}\pi_{ij} &\le \varepsilon,\\
  \sum_{i=1}^{n}\pi_{ij} &= 1/n \quad (j \in N),\\
  \sum_{j=1}^{n}\pi_{ij} &= \hat{p}_{i} \quad (i \in N).
\end{align*}
Here, for any feasible solution $(\vt{x}, \alpha, \vt{\mu})$, it holds that $-\log x_j \le \mu_i + c_{ij}\alpha$ for $i, j \in N$. Multiplying both sides of this inequality by $\pi_{ij} \ge 0$ and taking the sum over $i$ and $j$, we obtain
\begin{align*}
\sum_{i=1}^{n}\sum_{j=1}^{n}\pi_{ij}(-\log x_{j})\le\sum_{i=1}^{n}\sum_{j=1}^{n}\pi_{ij}\mu_{i}+\sum_{i=1}^{n}\sum_{j=1}^{n}c_{ij}\alpha\pi_{ij}=\sum_{i=1}^{n}\hat{p}_{i}\mu_{i}+\alpha\sum_{i=1}^{n}\sum_{j=1}^{n}c_{ij}\pi_{ij}\le\sum_{i=1}^{n}\hat{p}_{i}\mu_{i}+\alpha\varepsilon.
\end{align*}

The right-hand side is exactly the objective function of~\eqref{eq:otdro-reform}. Looking at the left-hand side, we have
\begin{align*}
  \sum_{i=1}^{n}\sum_{j=1}^{n}\pi_{ij}(-\log x_j)=\sum_{j=1}^{n}(-\log x_j)\left(\sum_{i=1}^{n}\pi_{ij}\right)=\sum_{j=1}^{n}(-\log x_j)(1/n)\ge-\log\left(\sum_{j=1}^{n}x_j/n\right)=-\log(1/n)=\log n,
\end{align*}
where the inequality follows from Jensen's inequality.
Thus, a lower bound on the objective function of~\eqref{eq:otdro-reform} is $\log n$.
By setting $(\vt{x}, \alpha, \vt{\mu}) = (\vt{1}_n/n, 0, (\log n)\vt{1}_n)$, the objective function of~\eqref{eq:otdro-reform} becomes $\varepsilon \cdot 0 + \sum_{i=1}^{n}\hat{p}_i \log n = \log n$. Therefore, the optimal value of~\eqref{eq:otdro-reform} is $\log n$, and the uniform distribution $\vt{1}_n/n$ is an optimal solution.
By Theorem~\ref{thm:otdro-equiv}, the uniform distribution $\vt{1}_n/n$ is therefore an
optimal solution of the OT-DRO model~\eqref{eq:otdro}.
\qed

\end{proof}

\bibliographystyle{spmpsci}
\bibliography{ref}

\end{document}